\documentclass[a4paper,10pt,centertags]{amsart}
\usepackage[utf8]{inputenc}
\usepackage{mathrsfs}
\usepackage{eulervm}
\usepackage{bbm}

\usepackage[T1]{fontenc}
\usepackage{tikz}
\usetikzlibrary{patterns,fadings}
\usepackage[colorlinks,pdfdisplaydoctitle]{hyperref}
\usepackage[a4paper,centering]{geometry}
\newtheorem{theorem}{Theorem}[section]
\newtheorem{lemma}[theorem]{Lemma}
\newtheorem{proposition}[theorem]{Proposition}
\newtheorem{corollary}[theorem]{Corollary}
\theoremstyle{definition}
\newtheorem{definition}[theorem]{Definition}

\theoremstyle{remark}
\newtheorem{remark}[theorem]{Remark}
\numberwithin{equation}{section}

\DeclareMathOperator{\Div}{div}
\DeclareMathOperator{\trig}{Trig}
\newcommand{\uap}{\mathsf{UAP}} 
\newcommand{\e}{\operatorname{e}}  
\newcommand{\im}{\mathrm{i}} 
\newcommand{\N}{\mathbf{N}}  
\newcommand{\Z}{\mathbf{Z}}  
\newcommand{\R}{\mathbf{R}}  
\newcommand{\C}{\mathbf{C}}  
\newcommand{\bes}{{\mathcal{B}}} 
\newcommand{\ste}{{\mathcal{S}}} 
\newcommand{\wie}{{\mathcal{W}}}
\newcommand{\mean}{\mathcal{M}}
\newcommand{\map}{\mathcal{I}}
\newcommand{\omean}{\overline{\mean}}
\newcommand{\loc}{\textrm{loc}}
\newcommand{\uloc}{\textrm{uloc}}
\newcommand{\scal}[1]{\langle #1 \rangle}   
\newcommand{\hatf}{\widehat{f}}
\newcommand{\tildeW}{\widetilde{W}}
\newcommand{\tildeU}{\widetilde{U}}
\newcommand{\uno}{\mathbbm{1}}
\newcommand{\apequiv}{\mathop{\overset{\textrm{\Tiny ap}}{\sim}}}
\newcommand{\lsum}[1][N]{\sideset{}{^{(#1)}}\sum}
\newcommand{\term}[1]{\text{\textcircled{\sf\footnotesize #1}}}
\newcommand{\ot}{\overline{t}} 
\newcommand{\hatw}{\widehat{w}} 
\newcommand{\hatpi}{\widehat{\pi}} 
\hypersetup{%
  pdftitle={On Leray's problem for almost periodic flows},
  pdfsubject={},
  pdfauthor={L. C. Berselli, M. Romito},
  pdfkeywords={Almost periodic flux, channel flow, Leray's problem},
  pdfcreator={M. Romito}
}
\begin{document}
  \title{On Leray's problem for almost periodic flows} 
  \author[L. C. Berselli]{Luigi C. Berselli}
    \address{Dipartimento di Matematica Applicata, Universit\`a di
      Pisa, Via F.~Buonarroti 1/c, I-56127 Pisa, Italia}
     \email{\href{mailto:berselli@dma.unipi.it}{berselli@dma.unipi.it}}
    \urladdr{\url{http://users.dma.unipi.it/berselli}}
  \author[M. Romito]{Marco Romito}
    \address{Dipartimento di Matematica, Universit\`a di Firenze,
      Viale Morgagni 67/a, I-50134 Firenze, Italia}
    \email{\href{mailto:romito@math.unifi.it}{romito@math.unifi.it}}
    \urladdr{\url{http://www.math.unifi.it/users/romito}}
  \thanks{This work has been partially supported by the GNAMPA project
    \emph{Modelli aleatorii e computazionali per l'analisi della
      turbolenza generata da pareti ruvide}} 
  \date{\today}
  \begin{abstract}
    We prove existence and uniqueness for fully-developed
    (Poiseuille-type) flows in semi-infinite cylinders, in the setting
    of (time) almost-periodic functions. In the case of Stepanov
    almost-periodic functions the proof is based on a detailed
    variational analysis of a linear ``inverse'' problem, while in the
    Besicovitch setting the proof follows by a precise analysis in
    wave-numbers.

    Next, we use our results to construct a unique almost periodic
    solution to the so called ``Leray's problem'' concerning 3D fluid
    motion in two semi-infinite cylinders connected by a bounded
    reservoir.  In the case of Stepanov functions we need a natural
    restriction on the size of the flux, while for Besicovitch
    solutions certain limitations on the generalized Fourier
    coefficients are requested.
  \end{abstract}
\maketitle
\tableofcontents
\section{Introduction}
%
We consider the motion of a viscous fluid in semi-infinite cylindrical
pipes, with an assigned (time) almost-periodic flux. The results are
aimed to find solutions of the so-called  ``Leray's problem.''
Moreover, this work can be considered as a intermediate step towards
the analysis of (deterministic) statistical solutions for the flow of
Poiseuille-type, which is the object of our current and ongoing
research. We recall that Leray's problem (which seems to have been
proposed by Leray to Lady\v{z}henskaya~\cite{Lad1959c,Lad1959}) is
that of determining a motion in a region with cylindrical exits,
subject to a given flux, and tending to the Poiseuille solution in
each exit. More precisely, let be given a connected open domain
$O\subset\R^3$ made of a ``reservoir'', a bounded and smooth open set
$O_0$, with two cylindrical exits $O_1$ and $O_2$.
\begin{figure}[ht]
  \begin{tikzpicture}[x=.9mm,y=1mm,line width=0.5pt]
    \fill [path fading=west,pattern color=blue!40!white,pattern=crosshatch dots]%
      (-10, 5) -- (15, 5) -- (15,10) -- (-10,10);
    \fill [path fading=east, pattern color=blue!40!white,pattern=crosshatch dots]%
      (70,15) -- (45,15) -- (45,20) -- (70,20);
     \def\upbasin{(15, 5) to [out=270,in=180] (23, 0) to [out=0,in=270]  (45,15)}
     \def\dnbasin{(45,20) to [out=90,in=0]    (37,25) to [out=180,in=90] (15, 10)}
     \fill [pattern color=blue!40!white,pattern=crosshatch dots] \upbasin -- \dnbasin;
    \draw [path fading=west,dashed] (-10, 5) -- ( 5, 5) (-10,10) -- ( 5,10);
    \draw [color=gray] (15, 5) -- (15,10);
    \draw ( 5, 5) -- (15, 5) ( 5,10) -- (15,10);
    \draw node at (8,7.5) {$O_1$};
    \draw [path fading=east,dashed] (70,15) -- (55,15) (70,20) -- (55,20);
    \draw [color=gray] (45,15) -- (45,20);
    \draw (45,15) -- (55,15) (45,20) -- (55,20);
    \draw node at (52,17.5) {$O_2$};
    \draw \upbasin \dnbasin;
    \draw node at (30,12.5) {$O_0$};
  \end{tikzpicture}
\end{figure}
These two semi-infinite exits (pipes) are described in coordinate systems
directed along the axis as
\[
  O_i=D_i\times \R^+,
\]
where the smooth cross sections $D_i$, $i=1,2$, may be possibly of
different shape and measure. We denote by $z\in \R^+$ the axial
coordinate in both cylinders.  Pioneering results in the stationary
case are those of Lady\v{z}henskaya~\cite{Lad1959b} and
Amick~\cite{Ami1977b}. See also the review in Finn~\cite{Fin1965}. The
extensive literature on the stationary problem is recalled for
instance in~\cite{Gal1994a,Gal1994b} for the linearized and full
Navier-Stokes problem, respectively. More recently the problem of
motion in pipes has also been addressed in the time-evolution case,
see Lady\v{z}henskaya and Solonnikov~\cite{LS1983} and also the
review in Solonnikov~\cite{Sol1992}. In the last decade Beir\~ao da
Veiga~\cite{Bei2005c} and Pileckas~\cite{Pil2006} gave new
contributions to the study of the time-dependent problem with assigned
flux, and the special role of the pressure has been also
emphasized by Galdi and coworkers~\cite{GR2005,GPS2007}.

In~\cite{Bei2005c} Leray's problem has been considered in the context
of time periodic flows, especially in view of application to the study
of blood flow and we recall that the role of blood flow in
mathematical research has been put in evidence by
Quarteroni~\cite{Qua2002}.  We also stress that the (non-trivial)
explicit solution introduced by Womersley is periodic, in some sense
generalizes the Poiseuille flow, and has been discovered in the study
of physiological flows. Since the heart is pumping with a flux which
is not periodic, but a superposition of possibly non-rational
frequencies, this suggests also to study the problem in the setting of
almost periodic functions.  This work has been originated by the
inspiring results in~\cite{Bei2005c} and especially from Remark~3
therein: The independence of the various constant on the period of the
flux let the author suggest about the possible extension to almost
periodic solutions. The problem nevertheless requires a precise
functional setting in order to detect the largest class of almost
periodic functions to be employed. Moreover, it seems that the
very-nice proof based on Fourier series in~\cite{Bei2005c} cannot be
directly applied to the new setting and in addition new difficulties
in treating the nonlinearities arise when almost periodic functions
are employed. This leads us to propose two different approaches in two
different functional settings. We finally remark that, in addition to
early results of Foias~\cite{Foi1962}, the approach \textit{via}
almost-periodic functions finds wide applications in fluid mechanics
(see for instance the recent paper by G{\'e}rard-Varet and
Masmoudi~\cite{GVM2010}).
\subsection{Setting of the problem}
The problem we wish to solve is to find a (time) almost-periodic
solution of the Navier-Stokes equations
\begin{equation}
  \label{eq:nse}
  \begin{cases}
    \partial_t u - \nu\Delta u + (u\cdot\nabla)\,u + \nabla p =
    0,&\qquad x\in O,\quad t\in\R,
    \\
    \Div u = 0, &\qquad x\in O,\quad t\in\R,
    \\
    u = 0&\qquad\text{on }\partial O,\quad t\in\R,
  \end{cases}
\end{equation}
such that $u$ converges in both pipes as $|z|\to\infty$ (in a sense we
shall make clear later) to the solution of the Poiseuille-type
problem. For clarity we recall (cf.\cite[\S~2]{Bei2005c}) that by
solution of the Poiseuille-type problem (of fully-developed flow) we
mean a solution of the Navier-Stokes equations such that, in a
reference frame with $z$ directed along the axis of the pipe and
$x:=(x_1,x_2)$ belonging to the orthogonal plane, is of the form
\[
  u(t,x,z) = (0,0,w(t,x))
    \qquad\text{and}\qquad
  p(t,x,z) = \pi(t,x,z)+p_0(t).
\]
Here $p_0(t)$ is an arbitrary function and in addition the flux
condition is satisfied
\[
  \int_{D} w(t,x)\,dx=f(t),
\]
for some given function $f$, where $D$ is the section of
the domain.\par
The Poiseuille-type \textit{ansatz} implies that Navier-Stokes
equations reduce in the semi-infinite pipes $O_i$, $i=1,2$, to the
following equations
\[
  \begin{cases}
    \partial_t w^i - \nu\Delta_x w^i +\partial_z p^i=0,
       &\qquad (x,z)\in O_i,\; t\in\R,
       \\
    \partial_{x_1} p^i=\partial_{x_2} p^i=0
       &\qquad (x,z)\in O_i,\; t\in \R,
       \\
    w^i(t,x) = 0,
       &\qquad x\in\partial O_i,\; t\in \R,
  \end{cases}
\]
where $\Delta_x$ denotes the Laplacian with respect to the variables
$x_1$ and $x_2$. From the first equation it follows that $\partial_z
p^i$ is independent of $z$. From the second equation, we also obtain
that $p^i$ is independent of $x$, hence $p^i(t,x,z)
=-\pi^i(t)\,z+p_0^i(t)$. Since the term $p_0^i(t)$ does not affect the
velocity field, we may assume that the pressure has the form
$p^i(t,z)=-\pi(t)^i\, z$. Moreover, the dependence of $w^i$ on the space
variables $x_1$ and $x_2$ allows us to consider a problem reduced to
the cross section $D_i$ of $O_i$, and the flux condition is
$\int_{D_i}w^i(t,x)\,dx=f(t)$.  This implies that we have to study in
each pipe the following problem (called in the sequel the ``basic
flow''): Find $(w^i(t,x),\pi^i(t))$ such that
\begin{equation}
  \label{eq:basic_flow}
  \begin{cases}
    \partial_t w^i(t,x) - \nu\Delta w^i(t,x) = \pi^i(t),&\qquad x\in
    D_i,\; t\in\R,
    \\
    w^i(t,x)=0&\qquad x\in \partial D_i,\; t\in\R,
    \\
    \int_{D_i} w^i(t,x)\,dx=f(t)&\qquad t\in\R,
  \end{cases}
\end{equation}
showing (under suitable assumptions) that if $f$ is almost periodic,
then the couple $(w,\pi)$ is almost-periodic, too.

We observe that, contrary to the stationary problem where the same
approach gives the well-known Poiseuille solutions, the solution of
the time-dependent motion is more complex for the determination of the
non-constant pressure (Observe that the classical Poiseuille solution
is that obtained for circular pipes, but nevertheless in more general
domains the same approach gives corresponding results). In our problem
for the basic flow we have two scalar unknowns and two equations, but
contrary to the classical problems in fluid mechanics one cannot get
rid of the quantity $\pi$ by means of projection operators. The
problem we have to solve can be considered as an \textit{inverse}
problem. Moreover the problem cannot be treated with the standard
variational tools in a direct way. We can write a single equation (the
``elimination'' of $\pi^i$ is obtained by taking the mean value over
$D_i$) obtaining
 \[
   \begin{cases}
    \partial_t w^i - \nu\Delta w^i +\nu\int_D\Delta w^i= f'(t),
    &\qquad x\in D_i,\; t\in\R, 
    \\
    w^i(t,x) = 0, &\qquad x\in\partial D_i,\; t\in \R,
  \end{cases}
\]
and the latter equation makes easy to understand why some knowledge
also of the derivative of the flux will be needed in order to
solve~\eqref{eq:pipe-global}. Moreover the usual energy-type estimates
obtained by testing with $w^i$, or with $-\Delta w^i$, and with
$w^i_t$ are not-conclusive when applied to this problem. In
particular, the lack of coercivity prevents from a direct application of
the standard techniques employed for parabolic problems,
see~\cite[Sec.~3]{Bei2005c}.  This particular issue has been addressed
in two different ways by Beir\~ao da Veiga~\cite{Bei2005c} (periodic
case) and Pileckas~\cite{Pil2006} (given smooth flux).

Even if we generalize to the almost-periodic setting the periodic
results obtained in~\cite{Bei2005c}, in the first part of the paper we
will mainly follow and suitably adapt the approach
of~\cite{Pil2006}. In particular we give special emphasis to the
solution of~\eqref{eq:basic_flow} since this represents one of the
main technical difficulties. The nonlinear problem is then treated by
means of perturbation arguments in a more or less standard way. We
want also to point out that in the huge literature on almost periodic
solutions we find particularly inspiring (for the choice of Stepanov
functions as suitable for our problem) the paper by Marcati and
Valli~\cite{MV1985} concerning compressible fluids.

In the second part of the paper we consider the problem in the larger
class of Besicovitch almost periodic solutions with an approach which
is more in the spirit of Fourier analysis. We give a different proof
of the existence of the basic flow which also covers the $H^1(\R)$
case and provides an alternative  proof of~\cite[Thm.~1]{Bei2005c},
when restricted to a time-periodic flux. The fully
nonlinear case needs, besides the natural assumption of large
viscosity, an additional assumption of regularity on the flux
(see~\eqref{eq:bes_phi1}) which accounts of the technical difficulties
of this case, due essentially to the non--local (in time) quantities
that are used, see Section~\ref{sec:final_considerations}.

\begin{remark}
  For its variational formulation and the use of energy estimates the
  problem seems to be naturally set in Hilbert spaces and this is not
  well fitting with the classical continuous (Bohr) spaces of almost
  periodic functions. A suitable choice of the spaces represents then
  a fundamental starting point. We are presenting two different proofs
  in two different settings, since they are substantial different and
  the assumptions we make on the flux are of very different nature.
  In the first part we deal with Stepanov a.p. functions, and the
  setting is much similar to the classical variational one for
  evolution partial differential equations.  In the second part we
  deal with Besicovitch a.p. functions and the proof use analysis in
  wave-numbers. We also point out that while the linear problem can be
  also treated in a unified way, for the nonlinear one the differences
  in the functional setting imply special assumptions on the size of
  the flux and on the Fourier coefficients, respectively.
\end{remark}

\bigskip

\noindent\textbf{Plan of the paper:} In Section~\ref{sec:Stepanov} we
consider the problem under the condition of a Stepanov almost periodic
flux. After recalling the main definition we give a complete solution
of Leray's problem, with the \textit{natural} (in space dimension
three) restriction of a large viscosity. As by-product of our
results, we also prove existence in the case of $H^1(\R)$ fluxes. In
Section~\ref{sec:Besicovitch} we consider the problem in the larger
class of Besicovitch almost periodic solutions and we prove existence
for the basic flow, together with existence for the nonlinear problem
under suitable restrictions on the flux.
\section{Leray's problem in the framework of Stepanov
  a.p. functions}\label{sec:Stepanov} 
Here we introduce a functional setting in which it is possible to
extend the result of~\cite{Bei2005c} to almost periodic solutions.
\subsection{Functional setting}
The problem of almost periodic solutions of partial differential
equations has been studied extensively in the last century, starting
with the work of Bohr, Muckenhoupt, Bochner, and Favard and many
others. See the review in Amerio and Prouse~\cite{AP1971},
Besicovitch~\cite{Bes1955}, Corduneanu~\cite{Cor1968}, and Levitan and
Zhikov~\cite{LZ1982}.\par
In the sequel we will use the standard Lebesgue $L^p$ and Sobolev
spaces $H^s=W^{s,2}$. For simplicity we also denote by $\|\,.\,\|$ the
$L^2$-norm. We will use the symbol $C$ to denote a generic
constant, possibly different from line to line, depending on the
domain and not on the viscosity $\nu$ or on the flux $f$. Next, given
a Banach space $(X,\|\,.\,\|_X)$ we denote by $\uap(\R;X)$ the space of
almost periodic functions in the sense of Bohr-Bochner. We recall that a
function $f\in C^0(\R;X)$ is almost periodic if and only if the set of
its translates is relatively compact in the
$C^0(\R;X)$-topology (observe that if $C^0_b(\R;X)$ denotes the space
of continuous bounded functions, then $\uap(\R;X)\subset C^0_b(\R;X)$). In
the context of weak and strong solutions to partial differential
equations it is probably better to work with a more general notion of
almost-periodicity, given for functions $f\in L^p_{\loc}(\R;X)$, which
is suited to deal with distributional solutions.
\begin{definition}[Stepanov $p$-almost periodicity]
  We say that the function $f:\,\R\to X$ is Stepanov $p$-almost
  periodic (denoted by $f\in \ste^p(\R;X)$) if $f\in L^p_{\loc}(\R;X)$
  and if the set of its translates is relatively compact in the
  $L^p_\uloc(\R;X)$ topology defined by the norm
  \[
    \|f\|_{L^p_\uloc(\R,X)}:=\sup_{t\in\R}\left[\int_t^{t+1}\|f(s)\|_X^p\,ds\right]^{1/p}.
  \]
  When $p=2$ we say simply that the function $f:\,\R\to X$ is Stepanov
  almost periodic.
\end{definition}

We will give the main result by using fluxes belonging to this class,
together with their first derivative. However, in the second part of
the paper we will consider also a wider class of almost periodic
functions: functions almost periodic in the sense of
Besicovitch. Further generalities (not needed in this section) on
almost periodic functions are given in
Section~\ref{sec:generalities_ap}.

A first main result that we will prove concerns the existence of the
``basic flow'' problem in this framework.
\begin{theorem}
  \label{thm:basic_flow_Stepanov}
  Let be given a smooth, connected, and bounded open set $D\subset
  \R^2$ and let be given $f$ such that $f,f'\in \ste^2(\R)$.  Then,
  there exists a unique solution $(w,\pi)$ of~\eqref{eq:pipe-global}
  such that
  \[
  \begin{aligned}
    & \Delta w, w_t\in \ste^2(\R;L^2(D)),
    \\
    &\nabla w\in \ste^2(\R;L^2(D))\cap C^0_b(\R;L^2(D)),
    \\
    &\pi\in \ste^2(\R),
  \end{aligned}
  \]
and 
\begin{equation}
  \label{eq:estimate-ulocR}
  \begin{aligned}
    \sup_{t\in\R} \ \Big[\nu\|\nabla w(t)\|^2+ \int_t^{t+1}\big(\nu^2\|\Delta
    w(s)\|^2+\|w_t(s)\|^2+|\pi(s)|^2\big)\,ds \Big]
    \\
    \leq
    C\big(\nu^2+1+\frac{1}{\nu}\big)\|f\|^2_{H^1_\uloc(\R)},
  \end{aligned}
\end{equation}
\end{theorem}%
%
%
%
\begin{remark}
  The result concerning the linear problem for the basic flow holds
  true in any space dimension. 
\end{remark}
This allows to obtain in a rather standard way the following result
for the Navier-Stokes equations.
\begin{theorem}
  \label{thm:Leray_Stepanov}
  Let $O$ as in the introduction and let be given $f$ such that
  $f,f'\in \ste^2(\R)$. There exists $\nu_0=\nu_0(f,O)\geq0$ such that
  if $\nu>\nu_0$ there exists a unique solution $u$ of~\eqref{eq:nse}
  such that
  \[
  u\in \ste^2(\R;H^s(O))\quad\text{for all }s<2
\]
and $u$ converges to a Poiseuille-type solution $w^i$ in each pipe, as
$|z|\to+\infty$.
\end{theorem}
 \begin{remark}
   The restriction on the viscosity is not surprising and is common to
   several results concerning the three-dimensional Navier-Stokes
   equations.  This is also observed in~\cite[Ch.~XI]{Gal1994b} since
   the existence of a flux carrier that can be absorbed by the
   dissipation for any positive viscosity is generally not known for
   cylindrical domains. This imposes (also in the stationary case)
   limitations on the size of the flux, in terms of the viscosity.

   We observe that also in the time-periodic case~\cite{Bei2005c}
   largeness (in terms of data of the problem) of the viscosity is
   required. Nevertheless, the results in~\cite{Bei2005c} concern weak
   solutions and uniqueness is not stated. On the other hand
   in~\cite{Pil2006} there is no restriction on the viscosity, since
   special ``\textit{two-dimensional-like}'' solutions are considered.
 \end{remark}
%
 \subsection{Construction of the solution of the ``basic
   flow''}\label{sec:basic_Stepanov}
%
In this section we give a detailed analysis of the existence of an
almost periodic basic flow and a complete proof of
Theorem~\ref{thm:basic_flow_Stepanov}. The problem is the following:
given $f,f\in \ste^2(\R)$ find a Stepanov almost periodic solution of
\begin{equation}
  \label{eq:pipe-global}
  \begin{cases}
    \partial_t w - \nu\Delta w = \pi, &\qquad x\in D,\; t\in\R,
    \\
    w(t,x) = 0, &\qquad x\in\partial D,\; t\in \R,
    \\
    \int_D w(t,x)\,dx = f(t) &\qquad t\in\R,
  \end{cases}
\end{equation}
\begin{remark}
  It is easy to check that one can analyze the slightly more general
  problem where~\eqref{eq:pipe-global} is replaced by
  \[
    \begin{cases}
      \partial_t w + \nu Aw = \pi e,	&\qquad t\in\R,
      \\
      \scal{w(t), e}_H = f(t)		&\qquad t\in\R.
    \end{cases}
  \]
  with an unbounded, linear, and with compact inverse operator $A$ on
  the Hilbert space $H$ with domain $D(A)$ and $e\in H$ with $e\not\in
  D(A)$ is given. Under suitable assumptions on $A$, the same
  procedure that we will employ can be used, see also~\cite{Bei2006b}.
  The same remark holds also for the results of
  Section~\ref{sec:Besicovitch}.
\end{remark}
We start by solving the following initial-boundary value problem in the
unknowns $(w,\pi)$,
\begin{equation}
  \label{eq:pipe}
  \begin{cases}
    \partial_t w - \nu\Delta w = \pi,
    &\qquad x\in D,\; t\in]0,T],
    \\
    w(t,x) = 0,
    &\qquad x\in\partial D,\; t\in ]0,T],
    \\
    \int_D w(t,x)\,dx = f(t),
    &\qquad t\in[0,T],
    \\
    w(0,x) = w_0(x),
    &\qquad x\in D.
  \end{cases}
\end{equation}
We follow essentially the same approach of~\cite{Pil2006}, with
additional care on the analysis of the initial datum and on the dependence
of the solution on the various parameters of the problem.
In the sequel we will employ a spectral (spatial) approximation using the
$L^2(D)$-orthonormal eigenfunctions $(e_k)_{k\in\N}$ of the Laplace operator,
\[
\begin{cases}
  -\Delta e_k=\lambda_k e_k
  &\qquad x\in D,
  \\
  e_k=0
  &\qquad x\in \partial D.
\end{cases}
\]
Define
\[
\beta_j:=(\uno,e_j),\qquad j\in\N,
\]
where $(\cdot,\cdot)$ denotes the $L^2(D)$ scalar product and $\uno$
is the function defined on $D$ such that $\uno(x)=1$ a.~e..  Without
loss of generality from now on we assume that $|D|$, the Lebesgue
measure of $D$, is equal to one.  Clearly,
\[
\uno=\sum_{k=1}^\infty \beta_j e_j(x)
\qquad\text{and}\qquad
\sum_{k=1}^\infty \beta_j^2=|D|=1.
\]
A special role is played by the pipe's flux carrier, i.~e., by the function
$\Phi$ (which belongs to $H^2(D)\cap H^1_0(D)$ under the smoothness
assumptions on $D$) defined as solution of the following Poisson problem
\begin{equation*}
  \begin{cases}
    -\Delta \Phi=\uno&\qquad x\in D,
    \\
    \Phi=0&\qquad x\in \partial D.
  \end{cases}
\end{equation*}
We define the following quantities 
\[
  \chi_0^2:=\int_D\Phi\,dx=\int_D|\nabla
  \Phi|^2dx>0\qquad\text{and}\qquad  \eta_0^2:=\int_D|\Phi|^2dx>0, 
\]
which clearly depend only on $D$.  Observe that the function $\Phi$
is enough in the stationary case to construct Poiseuille-type flows,
since in that case the problem for the flux and that for the pressure
completely decouple. On the other hand, in the time-dependent case the
situation if more complex, since both unknown depend also on the time.

To work with in our problem with a general viscosity we need the
scaled version of the flux carrier $\varphi:=\nu^{-1}\Phi$ which solves 
\begin{equation}
  \label{eq:1}
  \begin{cases}
    -\nu\Delta \varphi=\uno&\qquad x\in D,
    \\
    \varphi=0&\qquad x\in \partial D,
  \end{cases}
\end{equation}
and such that
\[
\int_D\varphi\,dx=\nu\int_D|\nabla \varphi|^2dx=
\frac{\chi_0^2}{\nu}>0\qquad\text{and}\qquad
\int_D|\varphi|^2dx=\frac{\eta_0^2}{\nu^2}. 
\]

We start our analysis by proving
the following result.
\begin{proposition}
  \label{prop:local_existence}
  Given $f\in H^1(0,T)$, assume that%
    \footnote{The initial condition here is chosen in such a way that the
      compatibility conditions on the flux at time $t=0$ are satisfied.}
  $w_0(x)=\tfrac{\nu\,\varphi(x)}{\chi_0^2}f(0)$. Then, there exists a
  unique solution $(w,\pi)$ of~\eqref{eq:pipe} such that
  \[
  \begin{aligned}
    &w\in C(0,T;H^1_0(D))\cap H^1(0,T;L^2(D))\cap L^2(0,T;H^2(D)),
    \\
    &\pi\in L^2(0,T),
  \end{aligned}
  \]
  satisfying the following estimate
\begin{equation}
  \label{eq:main_dependence-nu}
  \begin{aligned}
    \nu\|\nabla w (t)\|^2+
    \nu^2\int_0^t\|\Delta w (s)\|^2+\int_0^t\|w _t(s)\|^2\,ds+\int_0^t|\pi (s)|^2\,ds
    \\
    \leq  
    C\int_0^t\big((1+\nu^2)|f(s)|^2+(1+\nu)|f'(s)|^2\big)\,ds,
  \end{aligned}
\end{equation}
  with a constant $C$ depending only on $D$ (and in particular
  independent of $T$).
\end{proposition}
\begin{proof}
  We start by constructing, with the Faedo-Galerkin method, a global
  unique approximate solution in $V_m = \operatorname{Span}\langle
  e_1,\dots,e_m\rangle$.  The first step is to approximate the initial
  condition.  Let $\varphi$ be the function introduced in~\eqref{eq:1}
  and write $\varphi=\sum_{k=1}^\infty\varphi_k e_k$, where the series
  converges in $H^2(D)$ and $\varphi_k=(\varphi,e_k)$. Hence, the
  projection of $\varphi$ over $V_m$ is given by
  \[
    P_m\varphi:=\sum_{k=1}^m\varphi_k e_k.
  \]
  In order to satisfy the flux condition also at time $t=0$ we set
  \[
    w^m(0,x):=\frac{f(0)P_m\varphi(x) }{\int_DP_m\varphi(x)\,dx}.
  \]
  Observe that, for large enough $m\in \N$, the approximate initial
  datum is well-defined. In fact, $P_m\varphi\to\varphi$ in $L^2(D)$
  and since $|D|<+\infty$ then $P_m\varphi\to\varphi$ in
  $L^1(D)$. Since $\int_D P_m \varphi\to \int_D
  \varphi=\nu^{-1}\chi_0^2>0$, there exists $m_0\in\N$ such that
  $\int_D P_m\varphi\not=0$ for all $m\geq m_0$. Moreover,
  $w^m(0,x)\to w_0(x)$ in $H^2(D)$, as $m\to+\infty$.

  We write Galerkin approximate functions
  \[
    w^m(t,x)=\sum_{k=1}^m c ^m_k(t)\e_k(x),
  \]
 and  we look for a couple $(w^m,\pi^m)$ such that
  \begin{equation}
    \label{eq:galerkin}
    \frac{d}{dt}(w^m,e_j)
      +\nu (\nabla w^m,\nabla e_j)
    = \pi^m(\uno,e_j)
    \qquad \text{for }j=1,\dots,m,
  \end{equation}
  and $\pi^m:(0,T)\to(0,T)$ chosen so that the flux condition
  \begin{equation}
    \label{eq:approximate_mean}
    \int_D w^m(t,x)\,dx = f(t)
    \qquad \forall\,t\in(0,T),
  \end{equation}
  is satisfied. The equality is meaningful since $f$ is a.~e.~equal to
  a continuous function.  In terms of Galerkin coefficients
  $(c_j^m)_{1\leq j\leq m}$ we have for the initial condition that
  \[
  c_j^m(0):=f(0)\frac{\varphi_j}{\sum_{j=0}^m \varphi_k \beta_k},
  \]
  while the system of ordinary differential equations reads as
  \[
    \frac{d}{dt} c_j^m(t)
      + \nu\lambda_j c_j^m(t)
    = \pi^m(t)\beta_j
    \qquad     \text{for } j=1,\dots,m,
  \]
  and the solution can be written as follows 
  \[
    c_j^m(t)
     =    c_j^m(0)\e^{-\nu\lambda_j t}
       + \beta_j\int_0^t \pi^m(s)\e^{-\nu\lambda_j (t-s)}\,ds
    \qquad      \text{for }j=1,\dots,m.
  \]
  To find the equation satisfied by $\pi^m$ we multiply the
  latter equality by $\beta_j$ and sum over $j=1,\dots,m$ to get
  \[
    f(t)=\sum_{j=0}^m \beta_j c_j^m(0)\e^{-\nu\lambda_j t}+
    \sum_{j=0}^m \beta_j^2\int_0^t \pi^m(s)\e^{-\nu\lambda_j(t-s)}\,ds.
  \]
  Finally, to obtain an integral equation for $\pi^m$, we
  differentiate with respect to  time deducing 
  \[
    f'(t)=-\nu\sum_{j=0}^m \lambda_j \beta_j c_j^m(0)\e^{-\nu\lambda_j t}+
    \sum_{j=0}^m \Big(\beta_j^2\pi^m(t)-\nu\lambda_j\int_0^t
    \beta_j^2\pi^m(s)\e^{-\nu\lambda_j(t-s)}\,ds\Big),
  \]
  and this yields the following \textit{Volterra integral equation of the second type}
\begin{equation}
    \label{eq:Volterra}
    \begin{aligned}
      \pi^m(t)-\nu\int_0^t\sum_{j=1}^m\frac{\lambda_j
        \beta^2_j}{|\beta|^2}\e^{\nu\lambda_j(t-s)}\pi^m(s)\,ds=\frac{1}{|\beta|^2}
      f'(t)+f(0)\frac{\sum_{j=1}^m\beta_j \varphi_j\e^{-\nu\lambda_j
          t}}{\sum_{k=1}^m\beta_k \varphi_k},
    \end{aligned}
  \end{equation}
  where $|\beta|^2:=\sum_{k=1}^m \beta_m^2$. For any fixed $m\in \N$
  the kernel of the integral equation~\eqref{eq:Volterra} is bounded
  for all $0\leq s\leq t$.  This is enough to infer that if $f\in
  H^1(0,T)$, then there exists a unique $\pi^m\in L^2(0,T)$ satisfying
  the integral equation~\eqref{eq:Volterra} and such that
  \[
    \|\pi^m\|_{L^2(0,T)}\leq C_m(\nu)\|f\|_{H^1(0,T)},
  \]
  for a constant $C_m(\nu)$ possibly depending on $m$ and also on
  $\nu$. Especially the dependence on $m$ is crucial, since we
  consider the problem at fixed viscosity, while we need uniform
  estimates in $m\in\N$ to employ the Galerkin method. In particular,
  the uniform estimate does not follow directly since the series
  defining the kernel for $s=t$, that is
  $\big(\sum_{k=1}^{m}\beta_k^2\big)^{-1}\sum_{k=1}^{m}\lambda_k \beta_k^2$
  does not converge for $m\to+\infty$, see~\cite{Pil2006} for further
  details.

  We need to find the \textit{a priori} estimate in a different way,
  but observe that, once we have constructed $\pi^m$, we can use it as a
  \textit{given external force} in the equation for the
  velocity~\eqref{eq:galerkin}. By using $w^m$ as test function, we
  obtain (with the Schwarz inequality and by
  using~\eqref{eq:approximate_mean}) that
  \[
    \begin{aligned}
      \frac{1}{2}\|w^m\|^2+\nu\int_0^t\|\nabla
      w^m\|^2&=\frac{1}{2}\|w^m_0\|^2+
      \int_0^t\pi^m(s)(\uno,w^m(s))\,ds
       \\
       &= \frac{1}{2}\|w^m_0\|^2+
      \int_0^t\pi^m(s) f(s)\,ds
      \\
      & \leq
       \|w_0\|^2+
      \frac{1}{2\epsilon}\int_0^t|f(s)|^2
      +\frac{\epsilon}{2}\int_0^t|\pi^m(s)|^2\,ds.
    \end{aligned}
  \]
  Observe also that $w^m_0\to w_0$ in $L^2(D)$ and that
  $\|w^m_0\|^2\leq 2\|w_0\|^2$ for large enough $m$ and in addition
    \[
    \|\nabla w^m_0\|^2\leq 2\|\nabla
    w_0\|^2\leq\frac{2\nu^2}{\chi_0^2}\|f\|^2_{H^1(0,T)}\|\nabla\varphi\|^2\leq
    2\|f\|^2_{H^1(0,T)}.
    \]
    This gives the first a-priori estimate showing that, for all
    $m\geq m_0$ there exists a unique solution $w^m\in
    L^\infty(0,T;L^2)\cap L^2(0,T;H^1)$. Since the bounds we obtain on
    the solution are not independent of $m$, they cannot be used to
    make directly the Galerkin method to work. We can prove even more
    regularity on $w^m$ by standard estimates. In fact, by using as
    test function $w^m_t$ in the system satisfied by $w^m$ we get
    \[
      \|w^m_t(t)\|^2+\frac{\nu}{2}\frac{d}{dt}\|\nabla
      w^m(t)\|^2=\pi^m(t)(\uno,w^m_t(t))=
      \pi^m(t) f'(t).
    \]
    Hence, an application of the Schwarz inequality gives
    \begin{equation}
      \label{eq:Galerkin-2}
      \frac{\nu}{2}\|\nabla    w^m(t)\|^2+
      \int_0^t\|w^m_t\|^2\,ds\leq
      \nu\|\nabla w_0\|^2+
      \frac{\epsilon}{2}\int_0^t|\pi^m(s)|^2\,ds+\frac{1}{2\epsilon}\int_0^t|f'(s)|^2\,ds.
    \end{equation}
    By using $-\Delta w^m$ as test function in the equation
    satisfied by $w^m$ we also obtain
    \begin{equation}
      \label{eq:Galerkin-4}
      \frac{1}{2}\|\nabla    w^m(t)\|^2+
      \frac{\nu}{2}  \int_0^t\|\Delta w^m(s)\|^2\,ds\leq
      \|\nabla w_0\|^2+
      \frac{1}{2\nu}\int_0^t|\pi^m(s)|^2\,ds.
    \end{equation}
    These estimates are enough to construct, for each fixed $m\in\N$,
    a unique solution $(w^m,\pi^m)$, which is smooth, say
    $H^1(0,T;L^2(D))\cap L^2(0,T;H^2(D))$.  Again the presence of
    $\pi^m$ in the right-hand side prevents from uniformity in $m$.

    This issue is solved by using a special test function and a couple of
    nested a-priori estimates in the following lemma.
    \begin{lemma}
      \label{lem:estimate-galerkin}
    There exists $M_0\in \N$ (larger or equal than $m_0$) such that
    for all $m>M_0$ it holds for all $t\in[0,T]$
    \begin{equation}
      \label{eq:a-priori-estimate}
      \begin{aligned}
        \nu\|\nabla w^m(t)\|^2+
        \nu^2\int_0^t\|\Delta w^m(s)\|^2+\int_0^t\|w^m_t(s)\|^2\,ds+\int_0^t|\pi^m(s)|^2\,ds
        \\
        \leq  
        C\int_0^t\big((1+\nu^2)|f(s)|^2+(1+\nu)|f'(s)|^2\big)\,ds,
      \end{aligned}
    \end{equation}
    with a constant $C$ depending only on $D$ (hence independent of $m$ and of $T$).
  \end{lemma}
  \begin{proof}
    We observe that the function $P_m\varphi$, where $\varphi$ is
    defined in~\eqref{eq:1}, turns out to be a legitimate test
    function for the Galerkin system. With integration by parts, we
    obtain
    \[
      (w^m_t,P_m\varphi)-\nu(w^m,\Delta P_m\varphi)=
      \pi^m(t)(\uno,\varphi+(P_m\varphi-\varphi)).
    \]
    By adding to both sides of the previous equality the quantity
    $-\nu(w^m,\Delta\varphi)=(w^m,\uno)=f(t)$ we obtain 
    \[
      (w^m_t,P_m\varphi)+f=\pi^m(\uno,\varphi)+\pi^m(\uno,P_m\varphi-\varphi)+
     \nu (w^m,\Delta(P_m\varphi-\varphi)).
    \]
    The last term from the right-hand side vanishes
    since $w^m\in V_m$, while $\Delta(P_m\varphi-\varphi)\in
    V_m^\perp$. Hence squaring the latter equality (remember that
    $(\uno,\varphi)=\nu^{-1}\chi_0^2$) we obtain 
    \[
      \frac{\chi_0^4}{\nu^2}\,|\pi^m(t)|^2\leq
      3\big(\|w^m_t(t)\|^2\|P_m\varphi\|^2+|f(t)|^2
      +|\pi^m(t)|^2\|P_m\varphi-\varphi\|^2\big).
    \]
    Next, since $P_m$ is a projection operator it follows that there
    exists $M_0\in\N$ such that
    $\|P_m\varphi-\varphi\|^2<\chi_0^4/(6\nu^2)$ for all $m\geq
    M_0$. Consequently, we can absorb in the left-hand side the term
    involving $\pi^m$ from the right-hand side. 
    Consequently, after integration over $(0,T)$, we get
    \begin{equation}
      \label{eq:Galerkin-3}
            \int_0^t|\pi^m(s)|^2\,ds\leq 
      C\int_0^t\big(\|w^m_t(s)\|^2+\nu^2|f(s)|^2\big)\,ds\qquad\forall\,
      m\geq M_0,
    \end{equation}
    with a constant C depending only on $D$ (\textit{via} $\chi_0$ and
    $\eta_0$).  Consider again~\eqref{eq:Galerkin-2} and use
    now~\eqref{eq:Galerkin-3} with $\epsilon=1/C$. It is possible to
    absorb the term involving $\pi$ from the right-hand side,
    obtaining
    \[
      \begin{aligned}
        &\frac{\nu}{2}\|\nabla w^m(t)\|^2 +\frac{1}{2}
        \int_0^t\|w^m_t(s)\|^2\,ds
        \\
        &\leq \nu\|\nabla w_0\|^2+{C}
        \int_0^t\nu^2|f(s)|^2+|f'(s)|^2\,ds
        \\
        &\leq C \int_0^t(1+\nu^2)|f(s)|^2+(1+\nu)|f'(s)|^2\,ds
      \end{aligned}
    \]
    This shows a uniform bound on $\|w^m_t\|_{L^2(0,T)}$. Hence, coming back
    again to~\eqref{eq:Galerkin-3} we also obtain that
    \begin{equation*}
    \int_0^t|\pi^m(s)|^2\,ds\leq C
    \int_0^t(1+\nu^2)|f(s)|^2+(1+\nu)|f'(s)|^2\,ds
  \end{equation*}
   Next, by using the bound obtained in~\eqref{eq:Galerkin-4} we also get
   \begin{equation*}
    \nu^2 \int_0^t\|\Delta w^m(s)\|^2\leq   C
    \int_0^t(1+\nu^2)|f(s)|^2+(1+\nu)|f'(s)|^2\,ds
 \end{equation*}
 By collecting all the estimates and with Young's inequality we
 obtain~\eqref{eq:a-priori-estimate}.
  \end{proof}

  With the above lemma we can conclude the proof of the existence
  result, since uniform bounds imply that there exists a couple $(w,\pi)$
  and a sub-sequence $\{m_k\}_{k\in\N}$ such that
\begin{equation*}
  \begin{aligned}
    & w^{m_k}\rightharpoonup w\qquad\text{ in
    }H^1(0,T;L^2(D))\cap L^2(0,T;H^2(D)),
    \\
    &w^{m_k}\overset{*}{\rightharpoonup } w \qquad\text{ in
    }L^\infty(0,T;H^1_0(D)),
    \\
    &\pi^{m_k}\rightharpoonup \pi\qquad \text{ in }L^{2}(0,T).
  \end{aligned}
\end{equation*}
The problem is linear and this implies that $(w,\pi)$ is a
distributional solution of~\eqref{eq:pipe} with the requested
regularity. Uniqueness follows again from linearity of the problem.
\end{proof}

Since the estimates on the norm of the solution are independent of $T$
(they depend just on $D$ and on $\nu$) the same argument shows that if
$f\in H^1(0,+\infty)$, one can study the problem with arbitrary $T>0$
and then a unique solution (with the same regularity) exists in
$[0,+\infty)$. More generally we have also the following result on the
whole real line, which is obtained by letting the initial time to go
to $-\infty$.
\begin{corollary}
  \label{cor:H^1}
  Let be given $f\in H^1(\R)$. Then, there exists a unique solution
  $(w,\pi)$ of~\eqref{eq:pipe} defined for all $t\in\R$ such that
  \[
  \begin{aligned}
    & w\in C_o(\R;H^1_0(D))\cap H^1(\R;L^2(D))\cap L^2(\R;H^2(D)),
    \\
    &\pi\in L^2(\R),
  \end{aligned}
  \]
  with the same bounds as before (here $C_o$ denotes the subspace of
  continuous functions vanishing at infinity).
\end{corollary}
For our purposes of studying almost periodic solutions it is important
to show that one has a global solution, with uniformly bounded gradients,
also if the force $f$ is not in $H^1(0,+\infty)$, but just in
$H^1_\uloc(0,+\infty)$, where
\[
  \|f\|_{H^1_\uloc(0,+\infty)}
    := \sup_{t\geq0}\Big( \int_t^{t+1}|f(s)|^2+|f'(s)|^2\,ds\Big)^{\frac12}.  
\]
In particular, the following result will be crucial for the rest of
the paper.
\begin{proposition}
  \label{prop:intermediate_basic_uloc}
  Let be given $f\in H^1_{\uloc}(0,+\infty)$, then the unique solution
  of~\eqref{eq:pipe} exists for all positive time and it satisfies
      \[
        \begin{aligned}
          & w\in C_b(0,+\infty;H^1_0(D))\cap
      H^1_\uloc(0,+\infty;L^2(D))\cap L^2_\uloc(0,+\infty;H^2(D)),
      \\
      &\pi\in L^2_\uloc(0,+\infty),
    \end{aligned}
  \]
with the estimate~\eqref{eq:estimate-uloc0inf} in terms of the data. 
\end{proposition}
\begin{proof}
  Generally this result is straightforward in presence of a standard
  parabolic problem. Since here we deal essentially with an inverse
  problem, we give a detailed proof, which is nevertheless
  obtained adapting the usual techniques typical of almost periodic
  solutions, see e.~g.~Amerio and Prouse~\cite{AP1971}. Observe that
  the estimate~\eqref{eq:a-priori-estimate} does not give a direct
  control of $\sup_{0<t<T}\|\nabla w(t)\|$ since the bound depends
  on $\|f\|_{H^1(0,T)}$ and consequently the $H^1$-norm of $w$ may
  become unbounded when $T\to+\infty$, if $f\not\in H^1(\R)$. We first
  prove that $\|\nabla w\|\in L^\infty(0,+\infty)$.

  First observe that Proposition~\ref{prop:local_existence} imply that
  there exists a unique solution
  \[
  \begin{aligned}
    & w\in C(0,+\infty;H^1_0(D))\cap
    H^1_\loc(0,+\infty;L^2(D))\cap L^2_\loc(0,+\infty;H^2(D)),
    \\
    &\pi\in L^2_\loc(0,+\infty),
  \end{aligned}
  \]
  hence the following calculations will be completely justified. By
  using the function $\varphi$ from~\eqref{eq:1} as test function we
  obtain with integration by parts the following identity
  \begin{equation*}
    (w_t,\varphi)-\nu(\Delta w,\varphi)=
    (w_t,\varphi)-(w,\nu\Delta\varphi)=\pi(\uno,\varphi).
  \end{equation*}
Hence, by recalling the definition of $\varphi$ and the flux condition we
get
\begin{equation*}
  \frac{\chi_0^2}{\nu}\pi(t)=(w_t(t,x),\varphi(x))+f(t)
\end{equation*}
and consequently taking the square and integrating over the
$[\xi,\tau]$ (for all couples $\xi\leq\tau$) we obtain that there
exists $C$ depending only on $D$ such that
\begin{equation}
  \label{eq:pressurexuloc}
  \int_\xi^\tau|\pi(s)|^2\,ds\leq
  C\int_\xi^\tau\|w_t(s)\|^2+\nu^2|f(s)|^2\,ds.
\end{equation}
Next, we test the equation satisfied by $(w,\pi)$ by $w$ and $\nu^{-1 }w_t$,
and by recalling that $\int_D w_t\,dx=\big(\int_D w\,dx\big)'=f'(t)$ we
obtain (for any $\epsilon,\eta>0$) the following differential
inequalities, a.e. $t\in(0,+\infty)$
\begin{equation*}
  \begin{aligned}
   & \frac{1}{2}\|w\|^2+\nu\|\nabla w\|^2\leq
    \frac{\epsilon}{2}|\pi|^2+\frac{1}{2\epsilon}|f|^2,
    \\
    &\frac{1}{2}\|\nabla w\|^2+\frac{1}{\nu}\| w_t\|^2\leq
    \frac{\eta}{2\nu}|\pi|^2+\frac{1}{2\eta \nu}|f'|^2.
  \end{aligned}
\end{equation*}
By choosing $\epsilon=(2\nu C)^{-1}$ and $\eta=(2 C)^{-1}$, where $C$
is the constant appearing in~\eqref{eq:pressurexuloc} and by
integrating over an arbitrary interval $[\xi,\tau]$ we get
\begin{equation}
  \label{eq:22july}
  \begin{aligned}
    &\frac{1}{2}\|w(\tau)\|_{H^1}^2+\int_\xi^\tau\big(\nu\|\nabla
    w(s)\|^2+\frac{1}{\nu}\|w_t(s)\|^2\big)\,ds
    \\
    &\qquad \leq
    \frac{1}{2}\|w(\xi)\|_{H^1}^2+\int_\xi^\tau\nu(2+C)|f(s)|^2+\frac{C}{\nu}|f'(s)|^2\,ds.
  \end{aligned}
\end{equation}
Hence, by using the Poincar\'e inequality and dropping the
non-negative term $\nu^{-1}\|w_t\|^2$ we finally obtain that there
exist $c_1,c_2>0$ and depending only on $D$ such that
\begin{equation}
  \label{eq:22july_bis}
  \begin{aligned}
    \|w(\tau)\|_{H^1}^2+c_1\nu\int_\xi^\tau\|    w(s)\|^2_{H^1}\,ds\leq
    \frac{1}{2}\|w(\xi)\|_{H^1}^2+c_2\int_\xi^\tau\nu|f(s)|^2+\frac{1}{\nu}|f'(s)|^2\,ds.
  \end{aligned}
\end{equation}
Suppose now that for a given $\ot\in[0,+\infty[$ it holds that
\begin{equation}
  \label{eq:amerio-1}
  \| w(\ot)\|^2_{H^1} \leq \| w(\ot+1)\|^2_{H^1}
\end{equation}
and rewrite~\eqref{eq:22july_bis} in the interval $[\ot,\ot+1]$ as
  follows:
  \begin{multline}
    \label{eq:intermediate2}
    \|w(\ot+1)\|^2_{H^1} - \|w(\ot)\|^2_{H^1}
    + c_1\nu\int_{\ot}^{\ot+1} \| w(s)\|^2_{H^1}\,ds\leq
    c_2\int_{\ot}^{\ot+1}\nu|f(s)|^2+\frac{1}{\nu}|f'(s)|^2\,ds.
  \end{multline}
  Since by hypothesis~\eqref{eq:amerio-1} the first term is
  non-negative we obtain in particular that
    \begin{equation}
      \label{eq:intermediate-estimate}
      c_1\nu\int_{\ot}^{\ot+1} \| w(s)\|^2_{H^1}\,ds\leq
      c_2\left(\nu+\frac{1}{\nu}\right)\|f\|^2_{H^1_\uloc}.
    \end{equation}
    Next, by using the estimate~\eqref{eq:intermediate-estimate} and
    the same argument we get, for each couple $\tau\leq\xi$ with
    $\tau,\xi\in[\ot,\ot+1]$, that 
    \[
    \begin{aligned}
      \left| \|w(\tau)\|^2_{H^1}-\| w(\xi)\|^2_{H^1}\right| &\leq
      c_1\nu\int_{\xi}^{\tau}\| w^m(s)\|^2_{H^1}\,ds+
      c_2(\nu+\frac{1}{\nu})\|f\|_{H^1_\uloc}^2
      \\
      &\leq 2c_2(\nu+\frac{1}{\nu})\|f\|_{H^1_\uloc}^2.
      \end{aligned}
  \]
  Since $\|w\|^2_{H^1}$ is a continuous function we can fix
  $\xi\in[\ot,\ot+1]$ such that
  \[
  \| w(\xi)\|^2_{H^1}=\min_{\ot\leq s\leq\ot+1} \|w(s)\|^2_{H^1}.
  \]
  We have then, using~\eqref{eq:amerio-1} and the definition of $\xi$,
  \[
  \begin{aligned}
    \| w(\ot)\|^2_{H^1}
    &\leq \|w(\ot+1)\|^2_{H^1}
    \\
    &\leq \Bigl|\| w(\ot+1)\|^2_{H^1} - \| w(\xi)\|^2_{H^1}\Bigr|
    + \| w(\xi)\|^2_{H^1}
    \\
    & =   \Bigl|\|w(\ot+1)\|^2_{H^1} - \| w(\xi)\|^2_{H^1}\Bigr|
    + \int_{\ot}^{\ot+1}\| w(\xi)\|^2_{H^1}\,ds
    \\
    &\leq \Bigl|\| w(\ot+1)\|^2_{H^1} - \|w(\xi)\|^2_{H^1}\Bigr|
    + \int_{\ot}^{\ot+1}\| w(s)\|^2_{H^1}\,ds.
  \end{aligned}
  \]
  Finally, from~\eqref{eq:intermediate2}
  and~\eqref{eq:intermediate-estimate} we obtain
  \[
  \| w(\ot)\|^2_{H^1}\leq   C\left(\nu+1+\frac{1}{\nu^2}\right)\|f\|_{H^1_\uloc}^2.
  \]
  Hence, for each $\ot\in(0,+\infty)$ such that~\eqref{eq:amerio-1} holds true
  we have that $\|w(\ot)\|^2_{H^1}$ is bounded uniformly. On the contrary, if
  \[
  \| w(\ot)\|^2_{H^1} > \| w(\ot+1)\|^2_{H^1},
  \]
  we can repeat the same argument on the ``window'' $[\ot-1,\ot]$. In
this way it is clear that if we are not able to find an interval
$[\ot-n-1,\ot-n]$ with $n<\ot-1$ such that 
\begin{equation*}
  \| w(\ot-n-1)\|^2_{H^1}
  \leq \| w(\ot-n)\|^2_{H^1}
\end{equation*}
we obtain, for some $\overline{n}\in\N$ such that $0\leq \ot-n<1$, the inequalities
\[
   \| w(\ot-n)\|^2_{H^1} > \dots > \|w(\ot)\|^2_{H^1},
\]
hence that 
\[
\| w(\ot)\|_{H^1}\leq  \sup_{0\leq t\leq1}\|w(t)\|^2_{H^1}.
\]
The latter is bounded simply by the local existence result. This
finally shows that, since $\|w_0\|_{H^1}$ depends itself on the norm
of $f$ in $H^1_\uloc(\R)$, for all $t\in\R^+$,
\[
\| w(t)\|^2_{H^1} + c_1 \nu \int_{t}^{t+1}\| w(s)\|^2_{H^1}\,ds \leq
C\left(\nu+1+\frac{1}{\nu^2}\right)\|f\|_{H^1_\uloc}^2.
\]
Next, by going back to~\eqref{eq:22july} we obtain an estimate on
$\|w_t\|^2_{L^2(t,t+1)}$, which implies by \eqref{eq:pressurexuloc}
the corresponding estimate for $|\pi|^2_{L^2(t,t+1)}$. Finally, by
comparison we get also an estimate for $\|\Delta w\|^2_{L^2(t,t+1)}$
and by collecting all inequalities we finally get
\begin{equation}
  \label{eq:estimate-uloc0inf}
  \begin{aligned}
    \sup_{t\geq0} \ \Big[\nu\|\nabla w(t)\|^2+ \int_t^{t+1}\big(\nu^2\|\Delta
    w(s)\|^2+\|w_t(s)\|^2+|\pi(s)|^2\big)\,ds \Big]
    \\
    \leq
    C\big(\nu^2+1+\frac{1}{\nu}\big)\|f\|^2_{H^1_\uloc(0,+\infty)},
  \end{aligned}
\end{equation}
ending the proof of the Proposition.
\end{proof}
We can finally construct a solution over the whole real line and we
have the following result.
\begin{proposition}
  \label{prop:basic_uloc}
  Let be given $f\in H^1_{\uloc}(\R)$, with 
\[
  \|f\|_{H^1_\uloc(\R)}:=\sup_{t\in\R}\Big[\int_t^{t+1}|f(s)|^2+|f'(s)|^2\,ds\Big]^{1/2},
\]
then the unique solution of~\eqref{eq:pipe-global} exists for all times and it
satisfies
\[
\begin{aligned}
  &w\in C_b(\R;H^1)\cap H^1_\uloc(\R;L^2)\cap L^2_\uloc(\R;H^2),
  \\
  &\pi\in L^2_\uloc(\R),
\end{aligned}
\]
satisfying the estimate~\eqref{eq:estimate-ulocR}.
\end{proposition}
\begin{proof}
  We start by observing that since $f\in H^1_{\uloc}(\R)$, then it
  follows that $f\in C_b(\R)$ (again by a.e. identification). It
  follows directly that $f\in C(\R)$, but the control of the maximum
  of $|f|$ is obtained as follows. We claim that $\sup_{x\in\R}
  |f(x)|\leq 2\|f\|_{H^1_\uloc}$.  In fact, for each couple of points
  $x,y\in\R$ such that $|x-y|\leq1$ it follows that
  \[
    |f(x)-f(y)|=\left|\int_x^y f'(s)\,ds\right|\leq\|f\|_{H^1_\uloc}.
  \]
  Suppose now \emph{per absurdum} that there exists $x_0\in\R$ such
  that $|f(x_0)|>2\|f\|_{H^1_\uloc}$. The previous inequality implies
  that
  \[
  |f(x)| > \|f\|_{H^1_\uloc} \qquad\text{for all }x\in[x_0,x_0+1],
  \]
  hence the contradiction
  \[
   \int_{x_0}^{x_0+1} |f(s)|^2\,ds
     > \|f\|_{H^1_\uloc}^2.
  \]
  This proves that the bound on $|f(x)|$ is true for all $x\in \R$.

  The proof of Proposition~\ref{prop:basic_uloc} is then obtained by
  using the previous results from
  Proposition~\ref{prop:intermediate_basic_uloc} to solve the following
  family of problems parametrized by $n\in\N$,
\[
  \begin{cases}
    \partial_t w_n - \nu\Delta w_n = \pi_n,	&\qquad x\in D,\; t>-n,
    \\
    w_n(t,x) = 0,				&\qquad x\in\partial D,\; t>-n,
    \\
    \int_D w_n(t,x)\,dx = f(t)			&\qquad t>-n,
    \\
    w_n(-n,x)=\frac{\nu\varphi(x)}{\chi_0^2}f(-n),&\qquad x\in D,\,t=-n.
  \end{cases}
\]
The same arguments as before imply that
\[
  \begin{aligned}
    &w_n\in L^2_\uloc(-n,+\infty;H^2(D))\cap
    H^1_\uloc(-n,+\infty;L^2(D)),
    \\
    &\nabla w_n\in L^\infty(-n,+\infty;L^2(D)), 
    \\
    &\pi_n\in L^2_{\uloc}(-n,+\infty),
      \end{aligned}
  \]
  with bounds independent of $n\in\N$ (the dependence on $\nu$ is the
  same as in~\eqref{eq:main_dependence-nu}).  By defining the extended
  functions $\widetilde{w}_n(t,x)$ and $\widetilde{\pi}_n(t,x)$ on the whole
  real line as
\[
    \widetilde{w}_n(t,x):=
    \begin{cases}
      w_n(t,x)&\quad t\geq -n,
      \\
      \frac{\nu\varphi(x)}{\chi_0^2}f(-n)&\quad t\leq -n,
    \end{cases}\quad\text{and}\quad 
    \widetilde{\pi}_n(t):=
    \begin{cases}
      \pi_n(t)&\quad t\geq -n,
      \\
      \frac{\nu f(-n)}{\chi_0^2}&\quad t<-n,
    \end{cases}
\]
we can extract a sub-sequence (relabeled) as
$(\widetilde{w}_n(t,x),\widetilde{\pi}_n(t,x))$ such that
\[
  \begin{aligned}
    &\widetilde{w}_n \rightharpoonup w
      \qquad&\text{in } L^2_\uloc(\R;H^2(D))\cap H^1_\uloc(\R;L^2(D)),
      \\
    &\nabla\widetilde{w}_n\overset{*}{\rightharpoonup} \nabla w
      \qquad&\text{in } L^\infty(\R;L^2(D)),\phantom{\hspace{3.6cm}}
      \\
    &\widetilde{\pi}_n\rightharpoonup \pi
      \qquad&\text{in } L^2_{\uloc}(\R).\phantom{\hspace{4.6cm}}
  \end{aligned}
\]
It is easy to see that $(w,\pi)$ is a distributional solution to~\eqref{eq:pipe}.
Since the problem is linear this is the unique solution. 
\end{proof}
We finally prove the result in Stepanov space of almost periodic
functions.
\begin{proof}[Proof of Theorem~\ref{thm:basic_flow_Stepanov}]
  Since the function $f$ is almost periodic, for each sequence
  $\{r_m\}\subset \R$ we can find a sub-sequence
  $\{r_{m_k}\}\subset\R$ and a function $\widehat{f}$ such that
  \[
    \sup_{t\in\R}\int_t^{t+1}\big(|f(\tau+r_{m_k})-\widehat{f}(\tau)|^2+
    |f'(\tau+r_{m_k})-\widehat{f}'(\tau)|^2\big)\,d\tau
      \overset{k\to+\infty}{\longrightarrow}0. 
  \]
  By using a standard argument by contradiction (see e.~g.~Foias and
  Zaidman~\cite{FZ1961} and Foias and Prodi~\cite{FP1967}) we assume
  \textit{per absurdum} that $w$ and $\pi$ are not almost periodic,
  hence that there exist a sequence $\{h_m\}\subset\R$, a function
  $\widehat{f}:\R\to\R$ such that
  \[
    \sup_{t\in \R}\int_t^{t+1}\big(|f(\tau+h_m)-\widehat{f}(\tau)|^2+
    |f'(\tau+h_m)-\widehat{f}'(\tau)|^2\big)\,d\tau
      \overset{m\to+\infty}{\longrightarrow}0,
  \]
  a constant $\delta_0>0$, and three
  sequences $\{t_p\}$, $\{h_{m_p}\}$, $\{h_{n_p}\}$ such that for
  all $p\in \N$,
  \[
    \sup_{t\in\R}\int_{t_p}^{t_p+1}\big(\nu^2\|\Delta
    w(\tau+h_{m_p})-\Delta w(\tau+h_{n_p})\|^2+
    |\pi(\tau+h_{m_p})-\pi(\tau+h_{n_p})|^2\big)\,d\tau\geq\delta_0.
  \]
  In addition, there are (by eventually relabeling the sequences) two real
  functions $\widehat{f}_1$ and $\widehat{f}_2$ such that
  \[
    \begin{aligned}
     & \sup_{t\in\R}\int_t^{t+1}\big(|f(\tau+t_p+h_{m_p})-\widehat{f}_1(\tau)|^2+
      |f'(\tau+t_p+h_{m_p})-\widehat{f}'_1(\tau)|^2\big)\,d\tau
        \overset{p\to+\infty}{\longrightarrow}0,
      \\
     & \sup_{t\in\R}\int_t^{t+1}\big(|f(\tau+t_p+h_{n_p})-\widehat{f}_2(\tau)|^2+
      |f'(\tau+t_p+h_{n_p})-\widehat{f}'_2(\tau)|^2\big)\,d\tau
        \overset{p\to+\infty}{\longrightarrow}0.
    \end{aligned}
  \]
  It is clear that $\widehat{f}=\widehat{f}_1=\widehat{f}_2$. We consider now the
  problem~\eqref{eq:pipe} with the two fluxes
  \[
    \begin{aligned}
     & F_{1 p}(t):=f(t+t_p+h_{m_p}),
      \\
      &F_{2 p}(t):=f(t+t_p+h_{n_p}),
    \end{aligned}
  \]
  and the corresponding  solutions
  \[
    \begin{aligned}
   & w_{1 p}(t,x):=w(t+t_p+h_{m_p},x),\qquad  \qquad   \pi_{1 p}(t,x):=\pi(t+t_p+h_{m_p}),
    \\
    &w_{2 p}(t,x):=w(t+t_p+h_{n_p},x),\qquad   \qquad   \pi_{2 p}(t,x):=\pi(t+t_p+h_{n_p}).
  \end{aligned}
  \]
Observe that $F_{1 p}(t)\to f(t)$ and $F_{2 p}(t)\to f(t)$ in $H^1_\uloc(\R)$
as $p\to\infty$. Hence, by passing to the limit as $p\to+\infty$ we construct
two solutions $(w_1,\pi_1)$ and $(w_2,\pi_2)$ corresponding to the same
flux $f$. In particular,
\[
  w_{1 p}\to w_1
    \qquad\text{and}\qquad
  w_{2 p}\to w_2,
\]
in $C_b(\R;H^1_0(D))$, but also in the topologies given in the statement of
Proposition~\ref{prop:basic_uloc}. Hence, we have in particular
\begin{multline*}
  \delta_0 \leq\int_{t_{p_s}}^{t_{p_s}+1}\nu^2\|\Delta
  w(\tau+h_{m_{p_s}})-\Delta w(\tau+h_{n_{p_s}})\|^2+
  |\pi(\tau+h_{m_{p_s}})-\pi(\tau+h_{n_{p_s}})|^2\,d\tau =
  \\
  = \int_{0}^{1}\nu^2\|\Delta w_{1 p_s}(\tau)-\Delta w_{2 p_s}(\tau)\|^2+
  |\pi_{1 p_s}(\tau)-\pi_{2 p_s}(\tau)|^2 \,d\tau \longrightarrow 0,
\end{multline*}
because the problem is linear and the two solutions $w_1$ and
$w_2$ corresponding to the same flux, coincide. This proves the
almost periodicity of $w$ and the same argument applied to $w_t$ ends
the proof of the result.
\end{proof}
\subsection{The full nonlinear problem in
  \texorpdfstring{$S^2(\R)$}{S2(R)}}\label{sec:nonlinear_Stepanov}
%
We now finally consider the Navier-Stokes equations and we look for
solutions of the following problem
\[
  \begin{cases}
    \partial_t u - \nu\Delta u + (u\cdot\nabla)u + \nabla p = 0,
      &\qquad x\in O,\,t\in\R,\\
    \nabla\cdot u = 0,
      &\qquad x\in O,\,t\in\R,\\
    u=0,
      &\qquad x\in \partial O,\,t\in\R,
  \end{cases}
\]
with
\[
  \sup_{t\in\R}\|u(t) - w_i(t)\|_{H^1_0(O_i)}\leq c_0,
\]
for $i=1,2$. Observe that this constraint implies
(see~\cite[\S~7]{Bei2005c} and~\cite[VI.I]{Gal1994a}) that
\[
\lim_{z\to+\infty}\|u(t) - w_i(t)\|_{H^{1/2}(D_i))}=0,
\]
uniformly in $t$.

As a preliminary remark we recall that it is well-known (see for
example~\cite[VI]{Gal1994a}, Amick~\cite{Ami1977b}) that due to the
particular shape of the unbounded domain $O$ the Poincar\'e inequality
holds true
\begin{equation}
  \label{eq:Poincare}
  \exists \,c_P>0\qquad   \|u\|_{L^2(O)}\leq c_P\|\nabla
  u\|_{L^2(O)}\qquad \forall\, u\in H^1_0(O) 
\end{equation}
and, if the boundary is smooth enough (say of class $C^{1,1}$), then
the following estimate for the Stokes operator holds true
\begin{equation*}
  \|u\|_{H^2(O)}\leq c\|P\Delta u\|_{L^2(O)},
    \qquad u\in H^2(O)\cap H^1_0(O).
\end{equation*}

We now take advantage of the results on the linear case of the
previous section, but first we need to ``glue'' the basic flows
constructed in the two pipes.  In this part of the paper we do not
claim any originality and we use a classical approach. We denote by
$z\in \R^+$ the axial coordinate in both cylinders and we define the
``truncated pipes''
\[
  O_i^r:=\left\{(x,z)\in O_i:\ z<r\right\}\qquad i=1,2,
\]
and we also define the truncated domain
\[
  O^r:=O_0\cup  O_1^r\cup   O_2^r.
\]
We first define a field (which is as smooth as $w_i$) defined on
$O$ and such that is equal to $w_i$ in the sets $O_i\backslash
O_i^1$. This extension is obtained by freezing the time variable and
by gluing together the functions $w_i$ by cut-off functions
$\phi_i(z)\in C^\infty(\R^3)$ depending just on the axial coordinate
$z$, and such that
\[
  \phi_i(z)=
  \begin{cases}
    1 &\qquad x\in O_i\backslash O_i^1,
    \\
    0&\qquad x\in  O_i^{1/2}.
  \end{cases}
\]
Next observe that, 
\begin{equation}
  \label{eq:V0_def}
  V_0:=\sum_{i=1}^2\phi_i\,w_i\in H^1_\uloc(\R;L^2(O))\cap L^2_\uloc(\R;H^2(O))
\end{equation}
since the mapping $(w_1,w_2)\mapsto V_0$ is bi-linear and the
extension does not involve the time-variable. 
\begin{remark}
  We also observe that if we are in a different functional framework
  the same procedure can be applied because the properties of the
  extension with respect to the time variable are the same of those of
  the basic flows $(w^i,\pi^i)$.
\end{remark}
This is not exactly the required extension, since the function $V_0$ is not
divergence-free.  To this end one has to use the Bogovski{\u\i}
formula to solve the \textit{linear} problem in the bounded domain
$O^1$: Find $\mathcal{B}\, V_0$ such that
\[
  \begin{cases}
    \nabla\cdot(\mathcal{B}\, V_0)=-\nabla \cdot V_0,&\qquad x\in O^1
  \\
 \mathcal{B}\, V_0=0,&\qquad x\in \partial O^1.
  \end{cases}
\]
Since $\nabla\cdot V_0\in H^1_0(O^1)$ and the compatibility
condition $\int_{O^1}\nabla\cdot V_0=0$ is satisfied the problem has a
solution $\mathcal{B}\,V_0\in H^2( O^1)$. Denoting again by
$\mathcal{B}\,V_0$ the null extension of $O^1$, we finally set
  \begin{equation}
   \label{eq:v0_def}
    w:=V_0+\mathcal{B}V_0.
  \end{equation}
The regularity of the solution of the divergence equation and the
previous argument shows that if $w_i\in H^1_\uloc(\R;L^2(O_i))\cap L^2_\uloc(\R;H^2(O_i))$,
then 
\[
w=w_i\qquad \text{in }O_i\backslash O^1_i
\]
and
\[
\|w\|_{H^1(\R;L^2(O))\cap L^2(\R;H^2(O))}
  \leq c\sum_{i=1}^2\|w_i\|_{H^1_\uloc(\R;L^2(O_i))\cap L^2_\uloc(\R;H^2(O_i))},
\]
for some $c=c(O_0,O_1,O_2)$.
We have finally the following result.
\begin{lemma}
  Let $(w_i,\pi_i)$ be solutions of the basic flow  with the
  regularity of Theorem~\ref{thm:basic_flow_Stepanov}. Then the
  extended function $w$ is Stepanov almost periodic and 
  \[
  \begin{aligned}
    & \|w\|_{L^\infty(\R;H^1_0(O)\cap H^1_\uloc(\R;L^2(O))\cap
      L^2_\uloc(\R;H^2(O))}
    \\
    &\qquad \leq c\sum_{i=1}^2\|w_i\|_{L^\infty(\R;H^1_0(O)\cap
      H^1_\uloc(\R;L^2(O_i))\cap L^2_\uloc(\R;H^2(O_i))}.
  \end{aligned}
  \]
\end{lemma}
With this result, we look now for solutions in the form 
\[
  u = U + w,
  \qquad
  p = P + \pi.
\]
We have to find $U\in L^\infty_\uloc(\R;L^2(O))\cap
L^2_\uloc(\R;H^1_0(O))$ solving
\begin{equation}
  \label{eq:auxiliar_Leray-nonlinear}
  \begin{cases}
    \partial_t U - \nu\Delta U
    +(U\cdot\nabla)\,U+(U\cdot\nabla)\,w+(w\cdot\nabla)\,U
    + \nabla P = F,
    \\
    \nabla\cdot U = 0,
    \\
    U = 0, &\text{on }\partial O,
    \\
    U(0,x)=-w(0,x).
  \end{cases}
\end{equation}
where 
\begin{equation}
  \label{eq:F_def}
  F(t,x,z):=-\big(\partial_t w(t,x)-\nu\Delta
  w(t,x)+(w(t,x)\cdot\nabla)\,w(t,x)\big)+ 
  \sum_{i=1}^2 \frac{\partial}{\partial z}(z\,\phi_i(z) \pi_i(t)).
\end{equation}
From the results of the previous section we can infer that $F\in
S^2(\R;L^2(O))$, hence we can now use standard results to show
existence of an almost periodic solution $U$.  The support of $F$ is
contained in the bounded subset $\mathcal{O}^1$, hence we can use the
standard variational techniques to show existence of a unique
solution, provided that the viscosity is large enough.  In particular
the only property that we need to check is that
$(w(t,x)\cdot\nabla)\,w(t,x)$ is almost periodic.  This follows since
if we take the $L^2(O)$-norm of $F$, this is equal to the
$L^2(O_1)$-norm. All other terms in the summation are clearly
relatively compact in $L^2_{\uloc}(\R)$, the nonlinear one can be
estimated as
\[
  \int_{t}^{t+1}\left|\int_{O_1}
    w(t,x)\cdot\nabla)\,w(t,x)\right|^2\leq
  \sup_t\|\nabla w(t)\|_{L^2(O_1)}^2\int_{t}^{t+1}\|\Delta w(s)\|_{L^2(O_1)}^2\,ds . 
\]
This proves that $F\in \ste^2(\R;L^2(O))$ and, by using this
expression and from the estimate~\eqref{eq:estimate-ulocR} on
$(w^i,\pi^i)$, we obtain that
\begin{equation*}
  \|F\|_{L^2_\uloc}^2\leq C\Big(\nu^2+1+\frac{1}{\nu^4}\Big)
  \|f\|_{H^1_\uloc}^2. 
\end{equation*}
In particular, as in the previous section, we need to show that there
exists a solution $U\in L^2_\uloc(0,+\infty;H^2(O))$, and this will
follow by using that the right-hand side is in
$L^2_\uloc(0,+\infty;L^2(O))$. The existence of a local solution
$L^2_\loc(0,+\infty;H^1_0(O))$ follows by standard arguments. In
particular, one has to perform a truncation in the space variables
(which is possible since $F$ is of compact support, see
again~\cite{Bei2005c}) and the usual a-priori estimates. In addition
to these rather standard results, we now prove that the solution is
strong (provided that the viscosity is large enough, which is
nevertheless needed also for weak solutions) and that the solution is
uniformly bounded in $H^1(O)$ and $L^2_{\uloc}(\R;H^2(O)).$ We present
just the a priori estimates, which can be justified by the usual
Galerkin method and truncation of the domain.

  \begin{proof}[Proof of Theorem~\ref{thm:Leray_Stepanov}]
    We multiply~\eqref{eq:auxiliar_Leray-nonlinear} by $-P\Delta U$
    and we carefully treat the various terms. First note that
    \[
      \begin{aligned}
        \int_{O_i} (P\Delta U)\cdot(U\cdot\nabla)\,U
          & =   \sum_k\int_k^{k+1}\int_{D_i} (U\cdot\nabla)\,U\, P\Delta U\,dx\,dz\\
          &\leq C\|U\|_{L^4((k,k+1)\times D_i)}\|\nabla U\|_{L^4((k,k+1)\times D_i)}
                  \|\Delta U\|_{L^2((k,k+1)\times D_i)}\\
          &\leq C\|U\|_{H^1(O_i)} \|\nabla U\|_{H^1(O_i)} \|\Delta U\|_{L^2(O_i)},
      \end{aligned}
    \]
    since the constant of the Sobolev embedding $H^1\subset L^4$ are
    uniformly bounded in each of the strips $(k,k+1)\times D_i$,
    see~\cite[Lemma~2.1]{Gal1994b}. Moreover the standard regularity
    theory for the Stokes operator shows also that
    \[
      \int_{O_i} (P\Delta U)\cdot(U\cdot\nabla)\,U
        \leq C\|U\|_{H^1(O_i)}\|P\Delta U\|_{L^2(O_i)}^2.
    \]
    The term $\int_{O_i} (U\cdot\nabla)\,w$ is then estimated as
    follows, by using the same splitting into slices of width 1 in the
    $z$-direction and the Sobolev embedding $H^2\subset L^\infty$,
    \[
      \begin{aligned}
        &\int_{O_i} (U\cdot\nabla)\,w\, P\Delta U\leq C
        \|w\|_{L^2(O_i)}\|P\Delta U\|_{L^2(O_i)}^2,
      \end{aligned}
    \]
    The other term is estimated in the same way, and the contribution
    from the domain $O_0$ is handled with standard tools. Hence adding
    together the three terms, we finally we arrive at the differential
    inequality
    \[
      \frac{d}{dt}\|\nabla U\|^2+\big(\nu-C_1(\|\nabla U\|+\|\nabla
      w\|)\big)\|P\Delta U\|^2\leq C_2\frac{\|F\|^2}{\nu}
    \]
    We fix the viscosity large enough, so that
    \[
      \nu-C_1(\|\nabla U(0)\| - \|\nabla w(0)\|)
        = \nu-2 C_1\|\nabla w(0)\|
        > 0.
    \]
    and this is possible since $\|\nabla
    w(0)\|=\frac{\|\nabla\varphi\|}{\chi_0^2}|f(0)|\leq
    \frac{C}{\nu}\|f\|_{H^1_\uloc}$. Since $\|F(t)\|\in L^2_\uloc(\R)$
    this is enough to show existence of a local unique solution, in a
    time interval $[0,\delta[$ for some positive $\delta$.

    The next step is to show that, under the same assumptions, the
    solution is global. This is rather standard, since if we have the
    uniform bound on $\|\nabla w\|^2$ coming from
    Proposition~\ref{prop:basic_uloc},
    \[
    \sup_{t\in\R}\|\nabla w(t)\|^2\leq
    C\big(\nu+1+\frac{1}{\nu^2}\big)\|f\|_{H^1_\uloc(\R)}^2, 
    \]
    we can fix $\nu_0=\nu_0(f,O)$ large enough such that
    \[
    \text{if }\ \nu>\nu_0 \qquad\text{then}\qquad \nu-C_1\|\nabla
    w(t)\|\geq\frac{\nu}{2}, \qquad t\in\R.
    \]
    For such $\nu$ we are reduced to solve (possibly redefining the constants)
    \begin{equation*}
      \frac{d}{dt}\|\nabla U\|^2+c_1\Big(\frac{\nu}{2}-\|\nabla
      U\|)\Big)\|\nabla U\|^2\leq \frac{C_2}{\nu}\|F\|^2 
    \end{equation*}
    Hence, if we define $Z(t)$ as the solution of the following Cauchy problem
    \begin{equation}
      \label{eq:auxiliar_Leray-nonlinear2}
      \begin{cases}
        Z'+c_1\big(\frac{\nu}{2}-\sqrt{Z}\big) Z= \frac{C}{\nu}\|F\|^2
        \\
        Z(0)=\|\nabla U(0)\|^2
      \end{cases}
    \end{equation}
    we have that $\|\nabla U(t)\|^2\leq Z(t)$. We employ now a
    fixed point argument in the space of function which are continuous
    and bounded over $[0,+\infty]$ to show that $Z$ is well defined
    and bounded in the same interval.  Let be given $\overline{z}\in
    C_b(0,+\infty)$ such that
    \[
      0\leq\overline{z}\leq\frac{\nu^2}{4},
    \]
    Solve now the problem
    \[
      \begin{cases}
        z'+c_1\big(\frac{\nu}{2}-\sqrt{\overline{z}}\big) z=
        \frac{c_2}{\nu}\|F\|^2.
        \\
        z(0)=Z(0)
      \end{cases}
    \]
    Since {$\frac{\nu}{2}-C_1 \sqrt{\overline{Z}}\geq\frac{\nu}{4}$} and
    both $z(0)$ and $ \frac{c_2}{\nu}\|F\|^2$ are non-negative, it
    follows that $z(t)\geq0$. Hence $z$ satisfies
    \begin{equation}
      \label{eq:Ascoli}
      z'+\frac{c_1\nu}{4} z\leq
      \frac{c_2}{\nu}\|F\|^2.
    \end{equation}
    The same argument employed in the proof of
    Proposition~\ref{prop:intermediate_basic_uloc} shows that
    \[
      0\leq z(t)\leq z(0)+ \frac{5 c_2}{\nu}\|F\|_{L^2_\uloc}^2\leq
      C\Big(\nu+1+\frac{1}{\nu^5}\Big)\|f\|_{H^1_\uloc}^2.
    \]
    Then, if $\nu$ is large enough such that
    \begin{equation*}
      C\Big(\nu+1+\frac{1}{\nu^5}\Big)\|f\|_{H^1_\uloc}^2\leq\frac{\nu^2}{4}. 
    \end{equation*}

    Consider now the map $\Phi:\overline{z}\mapsto z$ and given
    $\tau>0$ define
    \[
      K_\tau:=\left\{\overline{z}\in C^0([0,\tau]):\ 0\leq
        \overline{z}\leq\frac{\nu^2}{4}\right\}. 
    \]
    Clearly we have $\Phi(K_\tau)\subseteq K_\tau$ and the map is
    relatively compact by Ascoli-Arzel\`a theorem since $z$ is
    Lipschitz continuous as solution of~\eqref{eq:Ascoli}. Hence
    $\Phi$ as a unique fixed point which is a solution
    to~\eqref{eq:auxiliar_Leray-nonlinear2}. Since $\tau>0$ is
    arbitrary this proves that $Z$ exists on the whole interval
    $[0,+\infty)$. A standard comparison argument shows that
    $\|\nabla U(t)\|^2\leq Z(t)\leq \frac{\nu^2}{4}$ for all positive $t$.
    
    This estimate implies, by using standard argument well established for
    the Navier-Stokes equations, that there exists a solution 
    \[
      U \in C_b(0,+\infty;H^1_0(O))\cap L^2_\uloc(0,+\infty;H^2(O))\cap
      H^1_\uloc(0,+\infty;L^2(O)) ,
    \]
    and due to the regularity proved the solution is unique.

    Next, one can construct a global solution, by solving the
    following family of problems in $[-n,\infty)\times O$,
    \[
      \begin{cases}
        \partial_t U_n - \nu\Delta U_n + (U_n\cdot\nabla)U_n + (U_n\cdot\nabla)w + (w\cdot\nabla)\,U_n + \nabla P_n = F,\\
        \nabla\cdot U_n = 0,\\
        U_n = 0,               \qquad\text{on }\partial O,\\
        U_n(-n,x) = -w(-n,x)   \qquad\text{in }O.
      \end{cases}
    \]
    Again by prolongation we define a velocity on the whole real line
    by
    \[
      \tilde{U}_n(t,x)=
      \begin{cases}
        U_n(t,x)\qquad t\geq -n,
        \\
        U(-n,x)\qquad t< -n.
      \end{cases}
    \]
    We can show that $\tilde{U}_n(t,x)$ converges to a solution
    $U$ such that
    \[
    U\in C_b(\R;H^1_0(O))\cap L^2_\uloc(\R;H^2(O))\cap
      H^1_\uloc(\R;L^2(O)).
    \]
    To end the proof we need to show that if the external force is
    almost periodic, then $U$ is almost periodic too. For this
    result we need a result of ``asymptotic equivalence,'' which is
    obtained as follows. Let us suppose the we have two solutions
    of~\eqref{eq:auxiliar_Leray-nonlinear} on the interval
    $[t_0,+\infty)$ corresponding to different initial data, but to
    the same external force. Then the difference
    $\widetilde{U}:=U_1-U_2$ satisfies
    \[
      \begin{cases}
        \partial_t\widetilde{U} - \nu\Delta\widetilde{U}
        +(U_1\cdot\nabla)\,U_1-(U_2\cdot\nabla)\,U_2+(\widetilde{U}\cdot\nabla)\,w+(w\cdot\nabla)\,\widetilde{U}
        + \nabla\widetilde{P} = 0,
        \\
        \nabla\cdot\widetilde{U} = 0,
        \\
        \widetilde{U} = 0, \qquad\text{on }\partial O,
        \\
        \widetilde{U}(t_0,x)= \widetilde{U}_0.
      \end{cases}
    \]
    By multiplication by $\widetilde{U}$ and usual integration by
    parts, we get
    \[
      \frac{1}{2}\frac{d}{dt}\|\widetilde{U}\|^2+\nu\|\nabla\widetilde{U}\|^2\leq\int_O
      (\widetilde{U}\cdot\nabla)\,(w+U_2)\widetilde{U}\,dx. 
    \]
    By using the same techniques employed before to handle unbounded
    domains we show that
    \[
      \frac{1}{2}\frac{d}{dt}\|\widetilde{U}\|^2+\big(\nu-C(\|\nabla
      w\|+\|\nabla U_2\|)\big)\|\nabla\widetilde{U}\|^2\leq 0.
    \]
    The uniform bounds on $w$ and $U_2$ in $H^1_0(O)$ imply that
    for large enough viscosity one has
    \[
    \frac{d}{dt}\|\widetilde{U}\|^2+\nu C_P\|\widetilde{U}\|^2\leq 0.
    \]
    This finally shows that
    \[
      \|\widetilde{U}\|^2\leq \|\widetilde{U}(t_0)\|^2\e^{-\nu
        C_P(t-t_0)},\qquad t\geq t_0. 
    \]
    The same argument employed in the proof of
    Theorem~\ref{thm:basic_flow_Stepanov} can be employed and, for
    each sequence $\{r_m\}\subset \R$ we can find a sub-sequence
    $\{r_{m_k}\}\subset\R$ and a function $\tilde{F}$ such that
    \[
      \sup_{t\in\R}\int_t^{t+1}
      \|F(\tau+r_{m_k})-\tilde{F}(\tau)\|^2\,d\tau\overset{k\to+\infty}{\longrightarrow}0.  
    \]
    Assume by contradiction that $U$ is not almost periodic, hence
    that there exist a sequence $\{h_m\}\subset\R$ and a function
    $\tilde{f}:\R\to\R$ such that
    \[
      \sup_{t\in
        \R}\int_t^{t+1}\|F(\tau+h_m)-\tilde{F}(\tau)\|^2\,d\tau\overset{m\to+\infty}{\longrightarrow}0,  
    \]
    and a constant $\delta_0>0$ and three sequences $\{t_p\}, \{h_{m_p}\}$
    and $\{h_{n_p}\}$ such that for all $p\in \N$,
    \[
      \sup_{t\in\R}\int_{t_p}^{t_p+1}\|U(\tau+h_{m_p})-U(\tau+h_{n_p})\|^2\,d\tau\geq\delta_0.
    \]
There exists  two real functions $\tilde{F}_1$ and $\tilde{F}_2$ such that
  \[
    \begin{aligned}
     &
     \sup_{t\in\R}\int_t^{t+1}\|F(\tau+t_p+h_{m_p})-\tilde{F}_1(\tau)\|^2\,d\tau\overset{k\to+\infty}{\longrightarrow}0. 
      \\
     &
     \sup_{t\in\R}\int_t^{t+1}\|F(\tau+t_p+h_{n_p})-\tilde{F}_2(\tau)\|^2\,d\tau\overset{k\to+\infty}{\longrightarrow}0. 
    \end{aligned}
  \]
  It  follows in a standard way that
  $\tilde{F}=\tilde{F}_1=\tilde{F}_2$ and we consider now the
  problem~\eqref{eq:auxiliar_Leray-nonlinear} with the two forces 
  \[
         F_{1 p}(t,x):=F(t+t_p+h_{m_p},x)
      \quad\text{and}\quad
      F_{2 p}(t,x):=F(t+t_p+h_{n_p},x),
      \]
  and the corresponding  solutions
  \[
    U_{1 p}(t,x):=U(t+t_p+h_{m_p},x)\quad\text{and}\quad
    U_{2 p}(t,x):=U(t+t_p+h_{n_p},x).
  \]
  By passing to the limit as $p\to+\infty$ we construct two
  solutions $U_1$ and $U_2$ corresponding to the same force
  $\tilde{F}$. In particular
  \[
    w_{1 p}\to w_1 \text{ in }C_{\loc}(\R;L^2),
      \qquad
    w_{2 p}\to w_2 \text{ in }C_{\loc}(\R;L^2).
  \]
  In addition we have 
  \begin{multline*}
    \delta_0
      \leq\int_{t_{p_s}}^{t_{p_s}+1}\|U(\tau+h_{m_{p_s}})-U(\tau+h_{n_{p_s}})\|^2\,d\tau =\\
      = \int_{0}^{1}\|U_{1 p_s}(\tau)-U_{2 p_s}(\tau)\|^2\,d\tau
      \longrightarrow \int_{0}^{1}\|U_{1}(\tau)-U_{2 }(\tau)\|^2\,d\tau.
  \end{multline*}
  On the other hand the asymptotic equivalence implies that
  $U_1=U_2$, since for all $t\in\R$,
  \[
    \|U_{1}(t)-U_{2 }(t)\|^2\leq   \|U_{1}(t_0)-U_{2 }(t_0)\|^2\e^{-\nu
      C_P(t-t_0)}
    \qquad t\geq t_0, 
  \]
  and letting $t_0\to-\infty$ we obtain a contradiction.

  To conclude the proof 
we show that the function $U\in \ste^2(\R;H^s(O))$ for all
$0\leq s<2$. In fact, take a sequence $\{r_n\}$ we can find a sub-sequence
    $\{r_{m_k}\}\subset\R$ and a function $\tilde{U}$ such that the
    sequence $\{U(\tau+r_{m_k})\}$ satisfies
    \[
      \sup_{t\in\R}\int_t^{t+1}
      \|U(\tau+r_{m_k})-\tilde{U}(\tau)\|^2\,d\tau\overset{k\to+\infty}{\longrightarrow}0.  
    \]
Hence the sequence $\{U(\tau+r_{m_k})\}$ is a Cauchy sequence in
$L^2_{\uloc}(\R;L^2(O))$, that is for every $\epsilon>0$ there is $N\in\N$
such that
\[
  \sup_{t\in\R}\int_t^{t+1} \|U(\tau+r_{m_p})-U(\tau+r_{m_s})\|^2\,d\tau\leq\epsilon
    \qquad p,s\geq N.
\]
Since $U \in H^1_\uloc(\R;L^2(O))\cap L^2_\uloc(\R;H^2(O))$, by
classical interpolation it follows that
\[
  U_{*}\in C(\R;H^{\sigma}(0,1;H^{2-2\sigma}(O))),
    \qquad 0\leq\sigma\leq1.
\] 
Hence we obtain that
\begin{multline*}
  \sup_{t\in\R}\int_t^{t+1} \|U(\tau+r_{m_p})-U(\tau+r_{m_s})\|^2_{H^s(O)}\,d\tau\leq\\
    \leq \sup_{t\in\R}\int_t^{t+1} \|U(\tau+r_{m_p})-U(\tau+r_{m_s})\|^{2-s}_{L^2(O)}
                  \|U(\tau+r_{m_p})-U(\tau+r_{m_s})\|^{s}_{H^2(O)}\,d\tau,
\end{multline*}
and by using H\"older inequality,
\[
  \begin{aligned}
    \int_t^{t+1} \|U(\tau+r_{m_p})-U(\tau+r_{m_s})\|^2_{H^s(O)}\,d\tau
      &\leq \left[\int_t^{t+1}\|U(\tau+r_{m_p})-U(\tau+r_{m_s})\|^{2}_{L^2(O)}\right]^{\frac{2-s}{2}}\cdot\\
      &\quad\left[\int_t^{t+1}\|U(\tau+r_{m_p})-U(\tau+r_{m_s})\|^{2}_{H^2(O)}\,d\tau\right]^{\frac{s}{2}}.
\end{aligned}
\]
Since 
\[
  \sup_{t\in\R} \int_t^{t+1}
    \|U(\tau+r_{m_p})\|^2_{H^2(O)}+\|U(\tau+r_{m_s})\|^{2}_{H^2(O)}\,d\tau\leq C,
\]
it follows that the sequence $\{U(\tau+r_{m_k})\}$ is a Cauchy sequence in
$L^2_{\uloc}(\R;H^2(O))$ as well, ending the proof.
 \end{proof}
\begin{remark}
The result with $f\in H^1(\R)$ follows in the same, even simpler, way.
\end{remark}
\section{Leray's problem in the framework of Besicovitch
  a.~p.~functions}
\label{sec:Besicovitch}
%
In this section we discuss the same problem in a more general
setting, and first we recall some definitions on almost periodic
solutions.
\subsection{Generalities on almost periodic
  functions}\label{sec:generalities_ap}
In the literature there are different definitions of \emph{almost
  periodic} functions and we need now to explain the precise setting
we are using.  We refer mainly to~\cite[ch.~I]{Bes1955} for further
details and references.  Let $\trig(\R)$ be the set of all
trigonometric polynomials, that is, $u\in\trig(\R)$ if there exist
$n\in\N$, $\xi_1,\dots,\xi_n\in\R$ and $u_1,\dots,u_n\in\C$ such that
\[
u(x) = \sum_{k=1}^n u_k\e^{\im\,\xi_k\,x},\qquad x\in\R.
\]
Next, a set $A\subseteq\R$ is \emph{relatively dense} if there exists
$L>0$ such that each interval of length $L$ contains an element of the
set $A$.
\begin{definition}[Bohr]
  \label{d:apbohr}
  A \emph{uniformly almost periodic function} ($\uap(\R)$) is a
  continuous function $f:\R\to\R$ such that there is a
  \emph{relatively-dense} set of $\varepsilon$-almost-periods. That
  is for all $\varepsilon>0$, there exist translations
  $T_\varepsilon>0$ of the variable $t$ such that
  \[
  |f(t+T_\varepsilon)-f(t)|\leq\varepsilon.
  \]
\end{definition}
It is easy to see that all trigonometric polynomials are almost
periodic according to the previous definition. Let $\uap(\R)$ be the
set of all uniformly almost periodic functions. Then $\uap(\R)$
coincides with the closure of $\trig(\R)$ with respect to the
\emph{sup-norm} $\|\cdot\|_{L^\infty}$. Alternatively, as recalled in
Section~\ref{sec:Stepanov}, a function $f\in \uap(\R)$ if the set
$\{f(\tau+\cdot):\tau\in\R\}$ of translates of $f$ is relatively
compact in $C(\R)$.  A more general notion of almost periodicity was
introduced by Stepanov in 1925.  To this end for $p\geq1$ and $r>0$,
define the norm
\[
\|f\|_{\ste^p,r} := \sup_{t\in\R}\Bigl(\frac{1}{r}\int_{t}^{t+r}
|f(s)|^p \,ds\Bigr)^{\frac1p}.
\]
Then, the space $\ste^p(\R)$ is the closure of $\trig(\R)$ with
respect to the norm $\|\cdot\|_{\ste^p,r}$ above. Notice also that
while the norm depends on $r$, the topology is independent of the
value of $r$, hence we re-obtain the definition used in
Section~\ref{sec:Stepanov} (Cf.~\cite{Bes1955}).

The definition was later extended by Weyl in 1927 by considering the
closure $\wie^p(\R)$ of trigonometric polynomials with respect to
the semi-norm
\begin{equation*}
\|f\|_{\wie^p} := \lim_{r\to\infty}\|f\|_{\ste^p,r}.
\end{equation*}
Finally, Besicovitch~\cite{Bes1955} defined the space $\bes^p(\R)$
as the closure of $\trig(\R)$ with respect to the \textit{semi-norm}
\[
\|f\|_{\bes^p} := \limsup_{R\to+\infty}\Bigl(\frac{1}{2R}\int_{-R}^R
|f(s)|^p\,ds\Bigr)^{\frac1p}.
\]
Notice that one can have $\|f\|_{\bes^p}=0$ even though
$f\not\equiv0$. For example this happens if $f$ is in $L^q(\R)$ with
$q>p$ or $f\in L^\infty(\R)$ and $|f(x)|\to0$ as $|x|\to\infty$. One
has the following strict inclusions
\begin{equation*}
\uap(\R) \subset\ste^p(\R) \subset\wie^p(\R)
\subset\bes^p(\R),\quad\text{for any } p\in]1,+\infty[,
\end{equation*}
(with obvious inclusions with different values of $p$). It turns out
that the spaces $\bes^p(\R)$ of Besicovitch almost periodic functions
are among the ``largest possible'' compatible with the treatment of
partial differential equations as we shall see in
Proposition~\ref{prop:riesz}.  Let us focus on the case $p=2$, since
$\bes^2(\R)$ has an Hilbert structure.  Given $f\in L^1_\loc(\R)$,
define
\[
\omean(f) := \limsup_{R\to\infty} \frac{1}{2R}\int_{-R}^R f(t)\,dt.
\]
The \emph{mean operator} $\mean(f)$ is defined as the above quantity
when the limit exists. Given $f\in\bes^1(\R)$, the \emph{(generalized)
  Fourier coefficients} of $f$ are defined as follows
\[
a_\lambda(f) := \mean(f(t)\e^{\im\lambda t})
\]
and the set
\[
\sigma(f) := \{\lambda\in\R: a_\lambda(f)\neq0\},
\] 
the \emph{spectrum} of $f$, is at most countable.

Define the equivalence relation $\apequiv$ as $f\apequiv g$ if
$a_\lambda(f-g) = 0$ for all $\lambda\in\R$. As stated above,
$|\cdot|_{\bes^p}$ is not a norm, as it can be zero on nonzero
functions. It turns out that the quotient space $\bes^p/\apequiv$ is a
Banach space (see~\cite{Bes1955}).  If $f,g\in\bes^2(\R)$ then the
mean $\mean(f g)$ is well-defined and $\mean(f\,
g)=\mean({f}_1\,{g}_1)$ if $f\apequiv{f}_1$ and $g\apequiv{g}_1$. The
space $\bes^2/\apequiv$ is an Hilbert space when endowed with the
scalar product $\scal{f,g}_{\bes^2}=\mean(f g)$ and we have the
fundamental result due to Besicovitch.
\begin{proposition}
  \label{prop:riesz}
  The exponential functions $t\mapsto\e^{\im\lambda t}$ are an
  orthonormal Hilbert basis for $\bes^2(\R)$. In different words, any
  $f\in\bes^2(\R)$ can be represented by its generalized Fourier
  series
   \begin{equation}
     \label{eq:generalized_Fourier}
  f(t)\apequiv\sum_{\lambda\in\sigma(f)}
  a_\lambda(f)\e^{\im\lambda t},
  \end{equation}
 and $\sum_\lambda |a_\lambda(f)|^2<\infty$.  

  Conversely, if one has a generalized series as above with square
  summable coefficients (as a generalized series), then there is a
  function $f\in\bes^2(\R)$ having the series as its own generalized
  Fourier series.
\end{proposition}
Hence, we can identify a function by means of its generalized series
as follows
\begin{equation}
  \label{eq:bes_parseval}
  \bes^2(\R) := \Big\{ f\in L^2_{loc}(\R):\ 
  \|f\|_{\bes^2}^2:=\sum_{\lambda\in\sigma(f)}|a_\lambda(f)|^2<+\infty
  \Big\}
\end{equation}
and the identification $f(t)\apequiv
\sum_{\lambda\in\sigma(f)}a_\lambda(f)\e^{\im \lambda t}$ holds in
the sense of convergence in $\bes^2(\R)$.

We turn to the pipe problem~\eqref{eq:pipe-global} in the \emph{almost
  periodic} case. Since the problem is linear, it is reasonable to
find a solution in terms of Fourier transform (or series). It is clear
that, once the problem is solved in the Fourier space, we are given
with coefficients $a(\xi)\in \C$ for $\xi\in\R$ and we are left with
the problem of reconstructing the solution by inverse Fourier
transforming. Since the convergence for classical Fourier series is
robust in $L^2$ for $\ell^2$ coefficients, likewise in the context of
almost periodic functions we consider the (correct) space $\bes^2$,
for which the analogous of the Riesz-Fischer theorems holds true.

In the following we shall also need spaces of the type
{Sobolev-Besicovitch}, which are defined in the following way.
\begin{definition}
  Given a real $s>0$, a function $f\in \bes^2(\R)$ belongs to
  $\bes^{s,2}(\R)$ if
  \[
  \|f\|_{\bes^{s,2}}^2 :=
  \sum_{\lambda\in\sigma(f)}(1+|\lambda|^2)^{s}|a_\lambda(f)|^2<+\infty.
  \]
\end{definition}
In particular, if $f\in\bes^{1,2}(\R)$, then the Fourier series for
the (formal) derivative $f'$ of $f$ is convergent and defines an
element of $\bes^2(\R)$.
\subsection{Construction of the ``basic flow'' in the Besicovitch
  setting}\label{s:bes_basic}
In this section we solve problem~\eqref{eq:pipe-global} in the
unknowns $(w,\pi)$, with the Besicovitch meaning. As a by-product
of the method we obtain also a new proof of the existence of the
basic flow in the periodic case.
\begin{theorem}
  \label{thm:bes_basic}
  Given $f\in\bes^{1,2}(\R)$ there are $w\in\bes^2(\R;H^2(D))$, with
  $\partial_t w\in\bes^2(\R;L^2(D))$, and $\pi\in\bes^2(\R)$ such
  that
  \[
  \partial_t w - \nu\Delta w\apequiv \pi
  \qquad\text{and}\qquad
  \int_D w(t,x)\,dx\apequiv f(t).
  \]
  Moreover, $w$ and $\pi$ are unique up to identification as almost
  periodic functions.  Finally, there exists $c>0$ such that
  \[
  \begin{gathered}
    \|\Delta w\|_{\bes^2(L^2)}\leq c\|f\|_{\bes^2} +
    \frac{c}{\nu}\|f'\|_{\bes^2},
    \\
    \|\pi\|_{\bes^2} + \|\partial_t w\|_{\bes^2(L^2(D))}\leq
    c\nu\|f\|_{\bes^2} + c\|f'\|_{\bes^2}. 
  \end{gathered}
  \]
\end{theorem}
In order to prove the theorem, we restate the problem by spectral
analysis in terms of Fourier transform with respect to the time
variable $t$ (with conjugate variable $\xi$). Once Fourier
transformed, problem~\eqref{eq:pipe-global} reads as follows: Find
$(\hatw,\hatpi)$ such that
\begin{equation}
  \label{eq:bes_basic_fou}
  \begin{cases}
    \im\, \xi\,\hatw(\xi,x) - \nu\Delta\hatw(\xi,x) = \hatpi(\xi), &\qquad
    x\in D,\ \xi\in\R,
    \\
    \hatw(\xi,x) = 0, &\qquad x\in\partial D,\ \xi\in\R,
    \\
    \int_D \hatw(\xi,x)\,dx = \hatf(\xi) &\qquad \xi\in\R.
  \end{cases}
\end{equation}
Clearly, the same result follows by a decomposition in Fourier series. 
The first equation yields $\hatw(\xi,x) = \hatpi(\xi) W_\xi(x)$, where
\[
W_\xi(x) := \bigl((\im\,\xi\,\mathrm{Id} - \nu\Delta)^{-1}\uno\bigr)(x)
\]
is defined to be the solution to the linear, stationary, and complex
system
\begin{equation}
  \label{eq:bes_basic_eqW}
  \begin{cases}
    \im \,\xi\, W_\xi - \nu\Delta W_\xi = 1,	&\qquad x\in D, 
    \\
    W_\xi(x) = 0,				&\qquad x\in\partial D,
  \end{cases}
\end{equation}
parametrized by $\xi\in\R$. 
Set
\[
a_\xi := \int_D W_\xi(x)\,dx,
\]
then by Fourier transforming the flux condition
in~\eqref{eq:pipe-global} we get $\hatf(\xi) = a_\xi\hatpi(\xi)$ and
in conclusion the solution to~\eqref{eq:pipe-global} (or, more
precisely, to~\eqref{eq:bes_basic_fou}), is given by
\begin{equation}
  \label{eq:bes_basic_formal}
  \hatpi(\xi) = \frac{1}{a_\xi}\hatf(\xi)
    \qquad\text{and}\qquad
  \hatw(\xi,x) = \frac{1}{a_\xi}\hatf(\xi)W_\xi(x).
\end{equation}
The problem reduces to analyse the behavior of the two terms $a_\xi$
and $W_\xi$ with respect to $\xi\in\R$. The main properties are
summarised in the following lemma.
\begin{lemma}
  \label{lem:bes_basic_props}
  For every $\xi\in\R$ it holds
  \begin{enumerate}
  \item $a_\xi = \nu\int_D |\nabla W_\xi(x)|^2\,dx - \im \xi\int_D
    |W_\xi(x)|^2\,dx$,
  \item $\nu\int_D \Delta W_\xi(x)\,dx = \im \xi a_\xi - 1$,
  \item $\nu^2\int_D |\Delta W_\xi(x)|^2\,dx + \xi^2\int_D
    |W_\xi(x)|^2\,dx = 1$.
  \end{enumerate}
\end{lemma}
\begin{proof}
  To prove the first property, take the complex conjugate of the
  equation satisfied by $W_\xi$, multiply by $W_\xi$ and integrate by
  parts, obtaining
  \[
  \begin{aligned}
    \int_D W_\xi(x)\,dx
    & = \int_D(-\im \xi|W_\xi|^2 - \nu W_\xi\Delta
    \overline{W_\xi})\,dx
    \\
    & = \nu\int_D |\nabla W_\xi(x)|^2\,dx - \im \xi\int_D |W_\xi(x)|^2\,dx.
  \end{aligned}
  \]
  For the second property, just integrate the equation for $W_\xi$ on $D$.
  In order to prove the third, take again the complex conjugate of the equation
  for $W_\xi$, but this time multiply by $\Delta W_\xi$. Next, integrate by
  parts and use the first two properties to get
  \[
  \begin{aligned}
    \im \xi a_\xi - 1
    & = \nu\int_D \Delta W_\xi\,dx
    = - \nu^2\int_D|\Delta W_\xi|^2\,dx + \im\nu\xi\int_D |\nabla
    W_\xi|^2\,dx
    \\
    & = - \nu^2\int_D|\Delta W_\xi|^2\,dx + \im \xi a_\xi - \xi^2\int_D |W_\xi|^2\,dx,
  \end{aligned}
  \]
  which proves the equality.
\end{proof}
Next, we need to understand the growth/decay of $a_\xi$ and $W_\xi(x)$
with respect to $\xi$, in order to show that the formal
expression~\eqref{eq:bes_basic_formal} defines a solution (in a
suitable sense).
\begin{lemma}
  The map $\xi\mapsto W_\xi$ is continuous on $\R$ with values in
  $H^2(D)\cap H_0^1(D)$. Moreover, as $|\xi|\to+\infty$ we have
  $W_\xi\to0$ in $H^2(D)$ and also
  \begin{equation}
    \label{eq:bes_basic_atinfty}
    \lim_{|\xi|\to +\infty} \xi\|\nabla W_\xi\|_{L^2}^2=0,
    \qquad\text{and}\qquad
    \lim_{|\xi|\to +\infty}  \xi^2\|W_\xi\|_{L^2}^2=1.
  \end{equation}
  Finally, the map $\xi\to a_\xi$ is continuous on $\R$ with
  values in $\R$ and 
  \[
  \lim_{|\xi|\to +\infty}  a_\xi= -\frac{1}{\xi}.
  \]
\end{lemma}
\begin{proof}
  Fix $\xi,\xi_0\in\R$, with $\xi\neq0$, and set
  $V:=W_{\xi_0}-W_\xi$. By symmetry we can assume that $\xi>0$ and
  start with the case $\xi_0>0$.  The new function $V$ solves $\im\xi
  V - \nu\Delta V = \im(\xi - \xi_0)W_{\xi_0}$.  Multiply by
  $\Delta\overline{V}$ and integrate by parts to get
  \[
  \im\xi\|\nabla V\|_{L^2}^2 + \nu\|\Delta V\|_{L^2}^2
  = \im(\xi - \xi_0)\int_D \nabla W_{\xi_0}\cdot\nabla\overline{V}\,dx.
  \]
  The imaginary part of the above formula yields
  \[
  \xi\|\nabla V\|_{L^2}^2 = (\xi - \xi_0)\Re\Bigl(\int_D\nabla
  W_{\xi_0}\cdot\nabla\overline{V}\,dx\Bigr)
  \leq \frac{\xi}2\|\nabla V\|_{L^2}^2 + \frac{(\xi -
    \xi_0)^2}{2\xi}\|\nabla W_{\xi_0}\|_{L^2}^2
  \]
  and so $\|\nabla V\|_{L^2}\leq\tfrac{|\xi - \xi_0|}{|\xi|}\|\nabla
  W_{\xi_0}\|_{L^2}$.  On the other hand the real part yields
  \begin{multline*}
    \nu\|\Delta V\|_{L^2}^2
    =    (\xi_0 - \xi)\Im\Bigl(\int_D\nabla
    W_{\xi_0}\cdot\nabla\overline{V}\,dx\Bigr)\leq
    \\
    \leq |\xi - \xi_0|\,\|\nabla W_{\xi_0}\|_{L^2} \|\nabla V\|_{L^2}
    \leq \frac{(\xi - \xi_0)^2}{|\xi|}\|\nabla W_{\xi_0}\|_{L^2}^2
  \end{multline*}
  and as $\xi\to\xi_0$, continuity follows.
  
  In the case $\xi_0=0$ the proof is slightly different, since one can
  prove directly taking the real part that
  \begin{equation*}
    \nu\|\Delta V\|_{L^2}^2
    \leq |\xi|\|\nabla W_{\xi_0}\|_{L^2}^2.
  \end{equation*}
  Since $V\in H^1_0(D)$, this proves that as $\xi\to 0$, then R $V$
  tends to zero in $H^2(D)$, but now with the order of
  $|\xi|^\frac{1}{2}$.
  
  Next, we consider the limit at $\infty$. By the previous lemma we
  know that $\nu\|\Delta W_\xi\|_{L^2}\leq1$, so there is a
  sequence $W_{\xi_n}$ converging weakly in $H^2(D)$ to some
  $W_\infty\in H^2(D)$. Indeed $W_\infty=0$, since by the second
  property of Lemma~\ref{lem:bes_basic_props} $\xi a_\xi$ converges to
  a finite limit (and in particular implies
  \eqref{eq:bes_basic_atinfty}) and so $a_\xi\to0$, in particular
  $W_\infty=0$.  Moreover, from this it follows that the whole
  function $W_\xi\rightharpoonup 0$ weakly in $H^2(D)$.
  Finally, from the third property of
  Lemma~\ref{lem:bes_basic_props} $\|\Delta W_\xi\|_{L^2}$ converges
  strongly to $0$.

The statement on $a_\xi$ follows in the same way, by using again
Lemma~\ref{lem:bes_basic_props}.
\end{proof}
\begin{remark}
  With a little more effort one can show that $W_\xi\to
  -\tfrac\im{\xi}$ point-wise as $|\xi|\to+\infty$, but we are not
  going to use this property in the sequel.
\end{remark}
\begin{proof}[Proof of Theorem~\ref{thm:bes_basic}]
  We use the identification~\eqref{eq:bes_parseval} of an almost
  periodic function with its Fourier series and of the $\bes^2$
  semi-norm with the sum of squares of Fourier coefficients (namely,
  Parseval's identity).  Consider $f\apequiv\sum_\xi
  \hatf(\xi)\e^{\im\xi t}$, then we only need to prove suitable bounds
  for the quantities $|a_\xi|$, $\|\Delta W_\xi\|_{L^2}$, and
  $\|W_\xi\|_{L^2}$.

  Indeed, solving problem~\eqref{eq:bes_basic_fou} for the Fourier
  components yields
  \[
  w(t,x) \apequiv \sum_\xi \frac{W_\xi(x)}{a_\xi} \hatf(\xi)\e^{\im\xi t}
  \qquad\text{and}\qquad
  \pi(t) \apequiv \sum_\xi\frac{\hatf(\xi)}{a_\xi}\e^{\im\xi t}.
  \]
  In order to capture the dependence of constants from $\nu$, we
  observe that $W_\xi = \tfrac{1}{\nu}\tildeW_{\xi/\nu}$, where
  $\tildeW_\xi$ is the solution to~\eqref{eq:bes_basic_eqW}
  corresponding to $\nu=1$. Set
  \[
  n_0(\xi) = \|\tildeW_\xi\|_{L^2}^2,\qquad
  n_1(\xi) = \|\nabla\tildeW_\xi\|_{L^2}^2,\qquad
  n_2(\xi) = \|\Delta\tildeW_\xi\|_{L^2}^2.
  \]
  Indeed,
  \[
  \begin{aligned}
    \|\Delta w\|_{\bes^2(\R;L^2)}^2
    & = \sum_\xi \|\Delta\hatw\|_{L^2}^2
    =   \sum_\xi \frac{1}{|a_\xi|^2}|\hatf(\xi)|^2\|\Delta
    W_\xi\|_{L^2}^2
    \\
    & =   \sum_\xi |\hatf(\xi)|^2 \frac{\nu^2 n_2(\frac\xi\nu)}
    {\nu^2 n_1(\frac\xi\nu)^2 + \xi^2 n_0(\frac\xi\nu)^2}
    \\
    &\leq c_1\|f\|_{\bes^2(\R)}^2 + \frac{c_2}{\nu^2}\|f'\|_{\bes^2(\R)}^2,
  \end{aligned}
  \]
  since for $|\xi|\leq\nu$,
  \[
  \frac{\nu^2 n_2(\frac\xi\nu)}{\nu^2 n_1(\frac\xi\nu)^2 + \xi^2 n_0(\frac\xi\nu)^2}
  \leq \frac{n_2(\frac\xi\nu)}{n_1(\frac\xi\nu)^2}
  \leq \max_{|\sigma|\leq1} \frac{n_2(\sigma)}{n_1(\sigma)^2}
  =    c_1,
  \]
  while for $|\xi|\geq\nu$,
  \[
  \frac{\nu^2 n_2(\frac\xi\nu)}{\nu^2 n_1(\frac\xi\nu)^2 + \xi^2
    n_0(\frac\xi\nu)^2} \leq
  \frac{1}{\nu^2}\xi^2\frac{n_2(\frac\xi\nu)}{\bigl[\bigl(\frac\xi\nu\bigr)^2
    n_0(\frac\xi\nu)\bigr]^2} \leq
  \frac{1}{\nu^2}\xi^2\sup_{|\sigma|\geq1}\frac{n_2(\sigma)}{\bigl(\sigma^2n_0(\sigma)\bigr)^2}
  = \frac{c_2}{\nu^2}\xi^2.
  \]
  Similarly,
  \[
  \begin{aligned}
    \|\partial_t w\|_{\bes^2(L^2)}^2 & = \sum_\xi
    \xi^2\|\hatw(\xi)\|_{L^2}^2 = \sum_\xi
    \frac{\xi^2}{|a_\xi|^2}|\hatf(\xi)|^2\|W_\xi\|_{L^2}^2
    \\
    & = \sum_\xi|\hatf(\xi)|^2 \frac{\nu^2\xi^2 n_0(\frac\xi\nu)}
    {\nu^2 n_1(\frac\xi\nu)^2 + \xi^2 n_0(\frac\xi\nu)^2}
    \\
    &\leq c_3\nu^2\|f\|_{\bes^2}^2 + c_4\|f'\|_{\bes^2}^2,
  \end{aligned}
  \]
  since for $|\xi|\leq\nu$,
  \[
  \frac{\nu^2\xi^2 n_0(\frac\xi\nu)}{\nu^2 n_1(\frac\xi\nu)^2 + \xi^2 n_0(\frac\xi\nu)^2}
  \leq \frac{\nu^2}{n_0(\frac\xi\nu)}
  \leq \nu^2\max_{|\sigma|\leq1}\frac{1}{n_0(\sigma)}
  =    c_3\nu^2,
  \]
  while for $|\xi|\geq\nu$,
  \[
  \frac{\nu^2\xi^2 n_0(\frac\xi\nu)}{\nu^2 n_1(\frac\xi\nu)^2 + \xi^2 n_0(\frac\xi\nu)^2}
  \leq \xi^2 \frac{1}{\bigl(\frac\xi\nu\bigr)^2 n_0(\frac\xi\nu)}
  \leq \xi^2\frac1{\inf\limits_{|\sigma|\geq1}\sigma^2 n_0(\sigma)}
  =    c_4\xi^2.
  \]
  Finally, with similar computations,
  \[
  \begin{aligned}
    \|\pi\|_{\bes^2}
    & = \sum_\xi |\hatpi(\xi)|^2
    =   \sum_\xi \frac{1}{|a_\xi|^2}|\hatf(\xi)|^2
    =   \sum_\xi |\hatf(\xi)|^2 \frac{\nu^4}{\nu^2 n_1(\frac\xi\nu)^2
      + \xi^2 n_0(\frac\xi\nu)^2}
    \\
    &\leq c_5\nu^2\|f\|_{\bes^2}^2 + c_4\|f'\|_{\bes^2}^2,
  \end{aligned}
  \]
  where $c_5 = \max_{|\sigma|\leq1}\tfrac{1}{n_1(\sigma)^2}$.
  The quantities $c_1$, \dots, $c_5$ are easily seen to be finite
  by the previous lemma.
\end{proof}
\begin{remark}
  The computations in the proof of Theorem~\ref{thm:bes_basic} provide
  an alternate proof to Theorem~1 in~\cite{Bei2005c}, as well as to
  Corollary~\ref{cor:H^1} (once (generalized) Fourier transform are
  replaced by Fourier series). Moreover, the following estimates hold,
  \[
  \begin{aligned}
    &\int_\R \|\Delta w\|_{L^2}^2\,dt \leq c\|f\|_{L^2(\R)}^2 +
    \frac{c}{\nu^2}\|f'\|_{L^2(\R)}^2,
    \\
    &\int_\R \|\partial_t w\|_{L^2}^2\,dt + \int_\R|\pi(t)|^2\,dt \leq
    c\nu^2\|f\|_{L^2(\R)}^2 + c\|f'\|_{L^2(\R)}^2,
    \\
    &\sup_{t\in\R}\|\nabla w(t)\|_{L^2}^2 \leq
    c\bigl(\nu\|f\|_{L^2(\R)}^2 +
    \frac{1}{\nu}\|f'\|_{L^2(\R)}^2\bigr).
  \end{aligned}
  \]
  Indeed, in the proof above we have shown that
  \begin{equation}
    \label{eq:bes_basic_stime}
    \begin{aligned}
      &\frac{\|\Delta W_\xi\|_{L^2}}{|a_\xi|}\leq
      c\max\big\{1,\nu^{-1}|\xi|\big\}
    \\
    &\frac{|\xi|\,\|W_\xi\|_{L^2}}{|a_\xi|}\leq
    c\max\big\{|\xi|,\nu\big\}
    \\
    &\frac{1}{|a_\xi|}\leq
    c\max\big\{|\xi|,\nu\big\}.
  \end{aligned}
\end{equation} 
Hence, by Parseval's identity, we get
\[
\int_\R\|\Delta w\|_{L^2}^2
=    \int_\R \frac1{|a_\xi|^2}|\hatf(\xi)|^2\|\Delta W_\xi\|_{L^2}^2
\leq c\|f\|_{L^2}^2 + \frac{c}{\nu^2}\|f'\|_{L^2}^2,
\]
and the other inequalities are obtained similarly. Finally, the
inequality for $\nabla w$ follows by integration by parts and the
identity $\frac{d}{dt}\|\nabla w\|^2_{L^2}=-2\int_D w_t\,\Delta w\,dx$.
\end{remark}
%
%
\subsection{On the meaning of the solution}
\label{ss:meaning}
We need to spend a few words about the notion of solution we
constructed. We observe that a given $f\in \bes^{1,2}(\R)$ is clearly
identified in the sense of $\bes^{1,2}(\R)$, hence by means of its
generalized Fourier series. This implies, for instance that if $w$ is
a solution in the sense of Besicovitch spaces, then $w+\overline{w}$
is also a solution, for any $\overline{w}\in L^2(\R;H^2)\cap
H^1(\R;L^2)$. This poses some restrictions to the interpretation of
the result. One would like to have some embedding in the space of
continuous functions in order to have a more precise identification of
the solution. A larger spaces in which we are able to solve the
equation is balanced by a weaker notion of solution.

In general one cannot expect the validity of the usual Sobolev
embeddings in $\bes^{s,p}(\R)$ as is explained for instance in
Pankov~\cite{Pan1990} and especially the identification with
$\uap(\R)$ functions is not a trivial fact. Classical counterexamples
can be found in the references cited, while the following general
embedding result is proved for instance in~\cite{IBDS1998}.
\begin{proposition}
  \label{prop:Sobolev-embedding}
  Let $\Xi\subset\R$ be countable and assume there is $\beta>0$ such
  that the generalized sum satisfies
  \[
  \sum_{\xi\in\Xi}\frac{1}{|\xi|^\gamma}\qquad
  \begin{cases}
    <+\infty\quad\text{for }\gamma>\beta
    \\
    =+\infty\quad\text{for }\gamma<\beta.
  \end{cases}
  \]
  If $\beta<2s$, then for every $f\in\bes^{s,2}(\R)$ such that
  $\sigma(f)\subset\Xi$, we have $f\in C^{r,\alpha}(R)\cap \uap(\R)$ for
  all $\alpha\in[0,s-r-\beta/2)$, where $r = \lceil
  s-\tfrac\beta2\rceil$ (with corresponding inequality for the norms).
\end{proposition}
\begin{remark}
  To simplify the notation from now on we denote by $f_\xi$ the
  (generalized) Fourier coefficient $\hatf(\xi)$ that is (more
  precisely) written as $a_\xi(f)$ in~\eqref{eq:generalized_Fourier}.
\end{remark}
 To understand this result, let us observe that if $c_\xi\in
\ell^1(\C)$, then the series $\sum_\xi c_\xi \e^{\im \xi t}$ converges
uniformly and can be identified with a continuous almost periodic
function $f\apequiv\sum_\xi c_\xi \e^{\im \xi t}$.

Moreover, for classical Fourier series, i.~e.~$\sigma(f)=\Xi\subseteq
\Z$, the $\beta$-condition is satisfied for $\beta=1$ and this shows
that if $(c_j)_{j\in\Z}$, $(j c_j)_{j\in\Z}\in \ell^2(\C)$, then
\[
\Bigl|\sum_{j\in\Z} c_j \e^{\im j t}\Bigr|^2 \leq \Bigl|\sum_{j\in\Z}
|c_j|\Bigr|^2 \leq \Bigl(\sum_{j\in\Z} j^2 |c_j|^2\Bigr)
\Bigl(\sum_{j\in\Z\setminus{0}}\frac{1}{j^2}\Bigr) < +\infty.
\]
This is the guideline to understand the result for $\bes^{s,2}(\R)$, since one
has -- roughly speaking -- to show an inequality similar to
\[
\Bigl|\sum_{\xi\in\sigma(f)} c_\xi \e^{\im \xi t}\Bigr|^2
\leq \Bigl|\sum_{\xi\in\sigma(f)} |c_\xi|\Bigr|^2
\leq \Bigl(\sum_{\xi\in\sigma(f)} |\xi|^{2s} |c_\xi|^2\Bigr) 
\Bigl(\sum_{\xi\in\sigma(f)} \frac{1}{|\xi|^{2s}}\Bigr)
<    +\infty.
\]
For instance Proposition~\ref{prop:Sobolev-embedding} implies the
following result.
\begin{corollary}
  Let be given $f\in\bes^{1,2}(\R)$, with $f\apequiv\sum_\xi
  f_\xi\e^{\im\xi t}$ such that
  \[
  \sum_{\xi\in\sigma(f)} \frac{1}{|\xi|^2} < +\infty.
  \]
  Then, there are $w\in\bes^2(\R;H^2(D))\cap \uap(\R;H^1_0(D))$, with
  $\partial_t w\in\bes^2(\R;L^2(D))$, and $\pi\in\bes^2(\R)$ such
  that~\eqref{eq:pipe-global} is satisfied in the sense of
  Besicovitch.
\end{corollary}
This makes also possible to consider the flux as
\[
f(t) = f_1(t) + f_2(t),
\]
with $f_1\in\bes^{1,2}(\R)$ and $f_2\in H^1(\R)$, so that
$\|f_2\|_{\bes^{1,2}}=0$. One can construct the solutions $w_1\in
\bes^{1,2}(\R;L^2(D))$ and $w_2\in W^{1,2}(\R;L^2(D))$ corresponding to $f_1$
and $f_2$ respectively, and add together.

This is not completely satisfactory, since we still do not have a
precise identification on the pressure. To this end one would like to
have a solution in the classical $\uap(\R)$ space for example. This can be
achieved by assuming stronger conditions on $f$ (rather than on its
spectrum), as shown by the following result.
\begin{proposition}
  \label{prop:regL1}
  Let $f$ be given, with $f\apequiv\sum_\xi f_\xi\e^{\im\xi t}$, such
  that
  \begin{equation*}
    \sum_\xi (1 + |\xi|) |f_\xi| < +\infty.
  \end{equation*}
  Then, there exists a unique solution $(w,\pi)$ to~\eqref{eq:pipe-global} such
  that
  \begin{equation}
    \label{eq:bes_basic_stimeL1}
    \begin{gathered}
      \sum_\xi \|\Delta w_{\xi}\|_{L^2}
      \leq c\sum_\xi |f_\xi| + \frac{c}{\nu}\sum_\xi |\xi|\,|f_\xi|,
      \\
      \sum_\xi |\pi_{\xi}| + \sum_\xi |\xi|\,\|w_{\xi}\|_{L^2}
      \leq c\nu\sum_\xi |f_\xi| + c\sum_\xi |\xi|\,|f_\xi|,  
    \end{gathered}
  \end{equation}
  where $w_{\xi}$ and $\pi_{\xi}$ are the (generalized) Fourier
  coefficients of $w$ and $\pi$, respectively. In particular, $w\in
  \uap(\R;H^2)$, $\partial_t w\in \uap(\R;H)$ and $\pi\in \uap(\R)$.
\end{proposition}
\begin{proof}
  First we notice that since $\sum_\xi |f_\xi|<+\infty$, then
  \[
  \sum_\xi |f_\xi|^2 \leq \Bigl(\sum_\xi |f_\xi|\Bigr)^2,
  \]
  showing that $f\in \bes^2(\R)$. The same argument shows also that
  $f'\in\bes^2(\R)$, and so Theorem~\ref{thm:bes_basic} ensures the existence
  of a unique solution. Since $w_{\xi} = a_\xi^{-1}{W_\xi} f_\xi$ and
  $\pi_{\xi}=a_\xi^{-1}f_\xi$, the estimates~\eqref{eq:bes_basic_stimeL1} follow
  immediately from~\eqref{eq:bes_basic_stime}.
  In order to show that $w$ and $\pi$ are Bohr-almost periodic, we consider
  a truncation 
  \[
  f_n = \sum_{\xi\in \sigma_n(f)} f_\xi\e^{\im\xi t},
  \]
  where $(\sigma_n(f))_{n\in\N}$ is an increasing sequence of finite
  subset of $\sigma(f)$ such that $\bigcup_n \sigma_n(f) = \sigma(f)$.
  For each $n\in\N$ we can consider~\eqref{eq:pipe-global} with flux given by the
  trigonometric polynomial $f_n$ and the
  estimates~\eqref{eq:bes_basic_stimeL1} imply uniform convergence of the
  corresponding solutions $(w_n,\pi_n)$ towards $\uap$ functions with
  the requested properties.
\end{proof}
%
\subsection{The nonlinear case}
\label{ss:bes_nonlin}
In this last section we consider the non-linear problem. Assume
preliminarily (we shall assume stronger assumptions on $f$ later) that
$f\in\bes^{1,2}(\R)$ and denote by $w_1$, $w_2$ the basic flows in the
two pipes $O_1$, $O_2$ respectively, provided by
Theorem~\ref{thm:bes_basic}. Let $V_0$ be the flow defined as
in~\eqref{eq:V0_def}, it is clear that $V_0$ is also almost periodic
and keeps the same regularity properties of $w_1$ and $w_2$, namely
\[
V_0\in\bes^2(\R;H^2(O)\cap H_0^1(O))\cap\bes^{1,2}(\R;L^2(O)),
\]
as well as the flow $w$ defined in~\eqref{eq:v0_def}. Indeed, both
flows are obtained by applying only linear operators in the space
variable to $w_1$ and $w_2$.

Consider the full nonlinear Leray's problem in the (Besicovitch)
almost periodic setting, namely to find a solution $(u, p)$ to the
problem
\begin{equation}
  \label{eq:bes_nl}
  \begin{cases}
    \partial_t u - \nu\Delta u + (u\cdot\nabla)\,u + \nabla p \apequiv
    0,
    \\
    \nabla\cdot u\apequiv 0,
  \end{cases}
\end{equation}
such that
\[
  \|u - w_i\|_{\bes^2(\R;H^1(O_i))}\leq c_0,
  \qquad i=1,2,
\] 
and this implies that
\begin{equation}
  \label{eq:bes_nl_close}
  \lim_{z\to+\infty}\|u - w_i\|_{\bes^2(\R;H^{1/2}(D_i))}=0.
\end{equation}

If $w$ is the flow defined in~\eqref{eq:v0_def}, consider the solution
$u = U + w$ as a perturbation of $w$. Consequently,
\begin{equation}
  \label{eq:bes_nl2}
  \begin{cases}
    \partial_t U - \nu\Delta U + (U\cdot\nabla)\,U + (U\cdot\nabla)\,w +
    (w\cdot\nabla)\,U + \nabla P \apequiv F,
    \\
    \nabla\cdot U\apequiv 0,
  \end{cases}
\end{equation}
where $F$ is defined in~\eqref{eq:F_def}. The main theorem of the
section is the following.
\begin{theorem}
  \label{thm:bes_nl}
  Assume that the flux $f$ satisfies
  \begin{equation}
    \label{eq:bes_phi1}
    \phi_*:=  \sum_\xi ( 1 + |\xi|) |f_\xi| < \infty.
  \end{equation}
  Then, there exists $\nu_0 > 0$, with $\nu_0 = \nu_0(f, O)$, such
  that for every $\nu\geq\nu_0$ problem~\eqref{eq:bes_nl} admits a solution
  \[
    u\in\bes^2(\R;H^1_0(O))\cap\bes^{\frac12,2}(\R;L^2(O)),
  \]
  which satisfies~\eqref{eq:bes_nl_close}.
\end{theorem}
The rest of the section is devoted to the proof of this result.
\subsubsection{Spectrum and module}
Before turning to the analysis of problem~\eqref{eq:bes_nl2}, we recall
that, since $f\in\bes^2(\R)$, its generalized Fourier series is
well-defined and its spectrum $\sigma(f)$ is the set of modes
$\xi\in\R$ corresponding to non-zero coefficients $f_\xi$ in the
Fourier expansion of $f$.

Since $f$ is real, it follows that $\overline{f_\xi}=f_{-\xi}$ and so
the spectrum is symmetric, namely $-\sigma(f) = \sigma(f)$.
\begin{definition}
  The set $\mu(f)$ is the $\Z$-\emph{module} of the spectrum of $f$,
  namely the smallest subset of $\R$ which contains $\sigma(f)$ and is
  closed for the sum (that is, if $\xi,\eta\in\mu(f)$, then $a\,
  \xi+b\,\eta\in\mu(f)$, for all $a,b\in\Z$).
\end{definition}
It is clear that $\mu(f)$ is also symmetric and, since $\sigma(f)$ is
at most countable, $\mu(f)$ is at most countable too.  Moreover, it is
easy to see that the spectra of $w$ and $F$, by linearity, are
contained in the spectrum of $f$. Indeed, by construction, the terms
$V_0$ and $w$, defined in~\eqref{eq:V0_def} and in~\eqref{eq:v0_def},
respectively have spectrum contained in $\sigma(f)$.

In the following, with the purpose of approximations, we shall need to
consider finite dimensional truncations. To this aim, we fix an increasing
sequence $(\mu_N(f))_{N\in\N}$ of subsets of $\mu(f)$ converging to $\mu(f)$,
that is $\mu_N(f)\subset\mu_{N+1}(f)$ and $\mu(f)=\bigcup_N\mu_N(f)$, and
such that $\mu_N(f) = - \mu_N(f)$.
%
\subsubsection{Reduction to a system in Fourier variables}
A remarkable feature of the nonlinearity we are going to analyse is
that if $u_1$, $u_2$ are Besicovitch almost-periodic, then
$(u_1\cdot\nabla)\,u_2$ is also in the same class. This result on
product of almost periodic functions is not true in general, but
as we will see in our case it holds since the spectrum of
the nonlinearity is contained in the module generated by $\sigma(u_1)$
and $\sigma(u_2)$. Having this in mind, we recast
problem~\eqref{eq:bes_nl2} in Fourier variables,
\begin{equation}
  \label{eq:bes_nl3}
  \im\xi U_\xi - \nu\Delta U_\xi
  + \sum_{\eta+\theta=\xi}\bigl[ (U_\eta\cdot\nabla)\,U_\theta +
  (U_\eta\cdot\nabla)\,w_\theta + (w_\eta\cdot\nabla)\,U_\theta\bigr]   +
  \nabla P_\xi = f_\xi, 
\end{equation}
for $\xi\in\mu(f)$, with $\nabla\cdot U_\xi = 0$, and the sum in the formula
above is extended over all $\eta,\theta\in\mu(f)$.

We shall use the following strategy to prove Theorem~\ref{thm:bes_nl}. We 
linearise the nonlinearity (by introducing an auxiliary field $\tildeU$) and 
solve the new linearised problem (in two steps, first for a finite number
of modes, then for all modes). The assumption on the viscosity allows to
have a uniquely defined map that gives a solution to the linearised problem
for each field $\tildeU$. The same assumption ensures that this map is a
contraction and its fixed point is the solution to problem~\eqref{eq:bes_nl}.
%
\subsubsection{Preliminary tools}
We prove two preliminary tools for the analysis of the problem.

We first consider the fields $w$ and $F$ defined respectively as
in~\eqref{eq:v0_def} and~\eqref{eq:F_def} and prove the following
estimates in terms of $\phi_*$.
\begin{lemma}
  \label{lem:bes_nl_tool1}
  Let be given $f\in\bes^{1,2}(\R)$, assume that~\eqref{eq:bes_phi1}
  holds. Then, there is $c>0$ (independent of $\nu$) such that
  \[
  \begin{gathered}
    \|\Delta w\|_{\bes^2(R;L^2)} \leq c\bigl(1 +
    \tfrac1\nu\bigr)\|f\|_{\bes^{1,2}(\R)},
    \\
    \|\pi\|_{\bes^2(\R)} +
    \|\partial_t w\|_{\bes^2(\R;L^2)}, \leq c(1 + \nu)\|f\|_{\bes^{1,2}(\R)}
    \\
    \sum_\xi\|\Delta w_{\xi}\|_{L^2} \leq c\bigl(1 +
    \tfrac1\nu\bigr)\phi_\star, 
    \\
    \sum_\xi \bigl(|\pi_{\xi}|
    + |\xi|\|w_{\xi}\|_{L^2}\bigr) \leq c(1 + \nu)\phi_\star,
    \\
    \|F\|_{\bes^2(\R;L^2)} \leq c\bigl((1+\nu) +
    (1+\tfrac1\nu)^2\phi_\star\bigr)\|f\|_{\bes^{1,2}(\R)},
    \\
    \sum_\xi\|f_\xi\|_{L^2} \leq c\bigl((1+\nu) +
    (1+\tfrac1\nu)^2\phi_\star\bigr)\phi_\star.
  \end{gathered}
  \]
\end{lemma}
\begin{proof}
  The inequalities for $w$, $\pi$ are a straightforward consequence of
  Theorem~\ref{thm:bes_basic}, Proposition~\ref{prop:regL1} and the
  definition~\eqref{eq:v0_def}. To prove the inequalities for $F$, we
  only need to consider the term $(w\cdot\nabla)\,w$ (the estimate of
  the other terms in $F$ follow from the estimates for $w$ and
  $\pi$). For $\xi\in\mu(f)$, writing the explicit expression
  for $  [(w\cdot\nabla)\,w]_\xi
    = \sum_{\eta+\theta=\xi} (w_\eta\cdot\nabla)\,w_\theta$, we obtain that
\[
\|[(w\cdot\nabla)\,w]_\xi\|_{L^2}
\leq c \sum_{\eta+\theta=\xi} \|\nabla w_\eta\|_{L^2}\|\Delta w_\theta\|_{L^2},
\]
hence
\[
\begin{aligned}
  \|(w\cdot\nabla)\,w\|_{\bes^2(\R;L^2)}^2 \leq
  \sum_\xi\Bigl(\sum_{\eta+\theta=\xi} \|\nabla w_\eta\|_{L^2}\|\Delta
  w_\theta\|_{L^2}\Bigr)^2 
  \\
  \leq \Bigl(\sum_\xi \|\nabla
  w_\xi\|_{L^2}\Bigr)^2\|\Delta w\|_{\bes^2(\R;L^2)}^2,
\end{aligned}
\]
and also 
\[
\begin{aligned}
  \sum_\xi\|[(w\cdot\nabla)\,w]_\xi\|_{L^2} \leq
  \sum_\xi\sum_{\eta+\theta=\xi} \|\nabla w_\eta\|_{L^2}\|\Delta
  w_\theta\|_{L^2} 
  \\
  \leq \Bigl(\sum_\xi \|\nabla
  w_{\xi}\|_{L^2}\Bigr)\Bigl(\sum_\xi\|\Delta w_{\xi}\|_{L^2}\Bigr),
\end{aligned}
\]
which complete the proof.
\end{proof}
The second result is the extension to our setting of the usual
cancellation property of the nonlinear convective term, when energy
estimates are derived.
\begin{lemma}
  \label{lem:bes_nl_tool2}
  Let be given $X=\sum X_\xi \e^{\im\xi t}$ and $X=\sum X_\xi
  \e^{\im\xi t}$ both belonging to $\bes^2(\R;H^1_0)$, with $Y$
  divergence-free.  Assume that the spectra of $X$ and $Y$ are
  contained in $\mu(f)$ and fix an integer $N\geq1$. Then
  \[
  \lsum_\xi\lsum_{\eta+\theta=\xi}\int_O
  \overline{X_\xi}\cdot(Y_\eta\cdot\nabla)\,X_\theta\,dx = 0,
  \]
  where the superscript on the sum above means that the sum is
  extended only over modes in $\mu_N(f)$. 

  Moreover, the same holds true for $N=\infty$ if at least one between
  $X$ and $Y$ is in $\bes^\star(\R;H^1_0)$, with
  \begin{equation*}
    \bes^\star(\R;H^1_0) :=\Big\{f\in L^1_{loc}(\R):\ 
    \|f\|_{\bes^\star(\R;H^1_0)}:=\sum_{\xi\in\sigma(f)}(1+|\xi|)\|f_\xi\|_{H^1}<+\infty\Big\}.
  \end{equation*}
\end{lemma}
\begin{proof}
  Let us denote by $\term{n}$ the sum in the statement of the lemma,
  then by a change of summation index
  \[
  \term{n} = \lsum_{\xi+\eta+\theta = 0}\int_O
  \overline{X_\xi}\cdot(Y_\eta\cdot\nabla)\,X_\theta\,dx = -
  \lsum_{\xi+\eta+\theta = 0}\int_O
  \overline{X_\theta}\cdot(Y_\eta\cdot\nabla)\,X_\xi\,dx, = - \term{n},
  \]
  since $\overline{X_\xi} = X_{-\xi}$, and the claim is true. To show
  that the same holds for $N=\infty$, it is sufficient to prove that
  the following sum extended is bounded uniformly in $N$. Indeed, if
  $Z$ is another field,
  \[
  \begin{aligned}
    \lsum_\xi\lsum_{\eta+\theta=\xi}\int_O
    \overline{X_\xi}\cdot(Y_\eta\cdot\nabla)\,Z_\theta\,dx &\leq
    \lsum_\xi\lsum_{\eta+\theta=\xi} \|\nabla X_\xi\|_{L^2} \|\nabla
    Y_\eta\|_{L^2} \|\nabla Z_\theta\|_{L^2}
    \\
    &\leq \|\nabla X\|_{\bes^2(H^1_0)} \|\nabla Y\|_{\bes^2(H^1_0)}
    \|Z\|_{\bes^\star(H^1_0)},
  \end{aligned}
  \]
  by Young's inequality for convolutions.
\end{proof}
We observe that in the cancellation property above it is fundamental
that we deal with the complex conjugate of $X$ and with the fact that
the spectrum is symmetric, since flux and solution are both
real--valued.
%

The proof of Theorem~\ref{thm:bes_nl} is split into three preliminary
steps.
\subsubsection{First step: existence for the finite modes
  approximation}
Given $\tildeU\in\bes^2(\R;H^1_0)$ with $\nabla\cdot\tildeU = 0$ and
an integer $N\geq1$, we seek for a solution to the following problem,
\begin{equation}
  \label{eq:bes_nl_aux}
  \im\xi U_\xi - \nu\Delta U_\xi
  + \lsum_{\eta+\theta=\xi}\bigl[ (\tildeU_\eta\cdot\nabla)U_\theta +
  (\tildeU_\eta\cdot\nabla)w_\theta +  (w_\eta\cdot\nabla)U_\theta\bigr] 
  + \nabla P_\xi = f_\xi,
\end{equation}
for $\xi\in\mu_N(f)$, where again the superscript on the above sum
means that the sum is extended only over modes in $\mu_N(f)$.
\begin{proposition}
  \label{prop:bes_aux}
  Let $f\in\bes^{1,2}(\R)$ and $\tildeU\in\bes^{1,2}(\R;H^1_0)$ and
  set
  \[
  \phi_\star^N := \lsum_\xi (1 + |\xi|)|f_\xi|\qquad\text{and} \qquad
  \psi_\star^N := \lsum_\xi \|\nabla\tildeU_\xi\|_{L^2}.
  \]
  Then, there is at least one solution $U^{(N)}$ to
  problem~\eqref{eq:bes_nl_aux}.  Moreover, there are non-negative
  $c_1(\cdot,\dots)$ and $c_2(\cdot,\dots,\dots)$ increasing functions
  of their arguments and (depending only on the domain $O$) such that
  \begin{equation}
    \label{eq:bes_nl_unif}
    \|U^{(N)}\|_{B^{\frac12,2}(\R;L^2)} + \|U^{(N)}\|_{B^2(\R;H^1)}
    \leq c(\nu,\|f\|_{\bes^{1,2}(\R)})\psi_\star^N + c(\nu,\phi_\star^N,\|f\|_{\bes^{1,2}(\R)}).
  \end{equation}
  If additionally
  \[
  \nu > \psi_\star^N + c(1 + \tfrac1\nu)\phi_\star^N,
  \]
  where $c$ is the constant of Lemma~\ref{lem:bes_nl_tool1}, then the solution
  is unique and
  \begin{equation}
    \label{eq:bes_nl_unif2}
    \begin{aligned}
      &\bigl(\nu - \psi_\star^N - c(1 +
      \tfrac1\nu)\phi_\star^N\bigr)\Bigl(\lsum_\xi \|\nabla
      U_\xi^{(N)}\|_{L^2}\Bigr)
      \\
      &\qquad\leq c(1+\nu)\phi_\star^N + c(1+\tfrac1\nu)^2(\phi_\star^N)^2 +
      c(1+\tfrac1\nu)\phi_\star^N\psi_\star^N.
    \end{aligned}
  \end{equation}
\end{proposition}
\begin{proof}
  The proof can be carried on with the standard technique of Fujita
  (cf.~Theorem 1.4 of~\cite[Ch.~2]{Tem2001}) for the case of existence
  of solutions for the steady Navier-Stokes equations in unbounded
  domains (but we have the additional advantage of the Poincar\'e
  inequality~\eqref{eq:Poincare}). We use Galerkin approximations $U^n
  = \sum_{k=1}^n u^k e_k$ (not necessarily made with eigenfunctions)
  and we consider the projection of~\eqref{eq:bes_nl_aux} on the
  finite dimensional Galerkin space as a problem on $\R^{2nN}$ (the
  real and imaginary parts of each $u^k$ count as two variables) with
  the scalar product induced by the one of $H^1_0(O)^{\otimes 2N}$.

  We show existence of a solution of the finite dimensional problem by
  means of Lemma 1.4 of~\cite[Chapter 2]{Tem2001} (which in turns is a
  consequence of Brouwer's fixed point theorem). Let $P(U^n)$ be given
  component-wise by the projection of~\eqref{eq:bes_nl_aux} so that if
  $P(U^n)=0$ then $U^n$ is the solution to the Galerkin projected
  problem. It is sufficient to show that $P(U^n)\cdot U_n >0$ on
  $|U^n| = C_0$ for some $C_0>0$, where product and norm are those we
  have given on $\R^{2nN}$. This is immediate by
  Lemma~\ref{lem:bes_nl_tool2} since
  \[
  \begin{aligned}
    P(U^n)\cdot U^n
    &\geq \nu|U^n|^2
    - \lsum_\xi\lsum_{\eta+\theta=\xi}|U^n_\xi|\|\nabla\tildeU_\eta\|_{L^2}\|\nabla w_\theta\|_{L^2}
    - c\|F\|_{\bes^2(L^2)}|U^n|
    \\
    &\geq \nu|U^n|^2
    - \psi_\star^N \|\nabla w\|_{\bes^2(L^2)} |U^n|
    - c\|F\|_{\bes^2(L^2)}|U^n|,
  \end{aligned}
  \]
  which is strictly positive if we choose $\nu$ such that
  \[
  C_0>
  \tfrac1\nu\bigl(\psi_\star^N \|\nabla w\|_{\bes^2(L^2)} + c\|F\|_{\bes^2(L^2)}\bigr).
  \]
  Passing to the limit in the Galerkin approximation is standard (the
  non-linearity contains a finite sum) and follows from uniform bounds
  (in $n$) on $U^n$ which are similar to~\eqref{eq:bes_nl_unif} and
  whose proof is formally similar.  Hence, we
  prove~\eqref{eq:bes_nl_unif} directly. For each $\xi\in\mu_N(f)$
  multiply~\eqref{eq:bes_nl_aux} by $\overline{U_\xi}$, integrate by
  parts on $O$, sum over $\xi\in\mu_N(f)$, and use
  Lemma~\ref{lem:bes_nl_tool2} to get
  \begin{multline*}
    \im\lsum_\xi \xi\|U_\xi\|_{L^2}^2
    + \nu\lsum_\xi \|\nabla U_\xi\|_{L^2}^2 +
    \\
    + \lsum_\xi\lsum_{\eta+\theta=\xi}\int_O
    \overline{U_\xi}\cdot(\tildeU_\eta\cdot\nabla)\,w_\theta\,dx 
    = \lsum_\xi\int_O \overline{U_\xi}\cdot f_\xi.
  \end{multline*}
  Hence by Young's inequality for convolutions and taking the real
  part we get
  \[
  \begin{aligned}
    \nu\|\nabla U\|_{\bes^2(L^2)}^2 &\leq
    \lsum_\xi\lsum_{\eta+\theta=\xi} \|\nabla U_\xi\|_{L^2}
    \|\nabla\tildeU_\eta\|_{L^2} \|\nabla w_\theta\|_{L^2} +
    c\|F\|_{\bes(L^2)}\|\nabla U\|_{\bes^2(L^2)}
    \\
    &\leq \psi_\star^N \|\nabla w\|_{\bes^2(L^2)} \|\nabla
    U\|_{\bes^2(L^2)} + c\|F\|_{\bes(L^2)}\|\nabla U\|_{\bes^2(L^2)},
  \end{aligned}
  \]
  and so inequality~\eqref{eq:bes_nl_unif} follows from Lemma~\ref{lem:bes_nl_tool1}.

  To prove~\eqref{eq:bes_nl_unif2}, multiply~\eqref{eq:bes_nl_aux} by
  $\overline{U_\xi}$, integrate by parts on $O$ and \emph{divide} by
  the non-zero $\|\nabla U_\xi\|_{L^2}$ to get
  \[
  \nu\|\nabla U_\xi\|_{L^2} \leq c\|f_\xi\| + \lsum_{\eta+\theta=\xi}
  \bigl(\|\nabla\tildeU_\eta\|_{L^2}\|\nabla U_\theta\|_{L^2} +
  \|\nabla\tildeU_\eta\|_{L^2}\|\nabla w_\theta\|_{L^2} + \|\nabla
  w_\eta\|_{L^2}\|\nabla U_\theta\|_{L^2} \bigr).
  \]
  Inequality~\eqref{eq:bes_nl_unif2} follows by summing in $\xi$, using
  Young's convolution inequality and Lemma~\ref{lem:bes_nl_tool1}.
  
  Finally, if $U_1^{(N)}$, $U_2^{(N)}$ are two solutions corresponding
  to the same data, let $D^{(N)}:= U_1^{(N)} - U_2^{(N)}$ and $Q^{(N)}:=
  P_1^{(N)} - P_2^{(N)}$, then
  \[
  \im\xi D^{(N)}_\xi - \nu\Delta D^{(N)}_\xi + \lsum_{\eta+\theta=\xi}
  \bigl(\tildeU_\eta\cdot\nabla)\,D^{(N)}_\theta +
  (w_\eta\cdot\nabla)\,D^{(N)}_\theta\bigr) + \nabla Q^{(N)}_\xi= 0,
  \]
  and hence taking the scalar product with $\overline{D^{(N)}_\xi}$
  \[
  \begin{aligned}
    \nu\|\nabla D^{(N)}\|_{\bes^2(\R;L^2)}^2
    &\leq \lsum_\xi\lsum_{\eta+\theta=\xi} \|\nabla w_\eta\|_{L^2}
    \|\nabla D^{(N)}_\xi\|_{L^2} \|\nabla D^{(N)}_\theta\|_{L^2}
    \\
    &\leq \Bigl(\lsum_\xi \|\nabla w_\eta\|_{L^2}\Bigr)\|\nabla
    D^{(N)}\|_{\bes^2(\R;L^2)}^2,
  \end{aligned}
  \]
  which, by the assumption on $\nu$, implies that $D\equiv0$.
\end{proof}
\subsubsection{Second step: existence of a limit as
  \texorpdfstring{$N\to\infty$}{N goes to infinity}} 
Let $f\in\bes^{1,2}(\R)\cap \bes^\star(\R)$ and assume additionally that the quantity
\[
\phi_\star := \sum_\xi (1+|\xi|)|f_\xi|
\]
is finite. This implies in particular, as in Lemma~\ref{prop:regL1},
that $f$ and $f'$ have representatives which are Bohr-almost
periodic.  Let $\tildeU\in\bes^2(\R;H^1_0)$, assume that the quantity
$\sum_\xi \|\nabla\tildeU_\xi\|_{L^2}$ is also finite and consider the
problem
\begin{equation}
  \label{eq:bes_nl_aux2}
  \im\xi U_\xi - \nu\Delta U_\xi
  + \sum_{\eta+\theta=\xi}\bigl[ (\tildeU_\eta\cdot\nabla)\,U_\theta +
  (\tildeU_\eta\cdot\nabla)\,w_\theta +  (w_\eta\cdot\nabla)\,U_\theta\bigr] 
  + \nabla P_\xi = f_\xi,
\end{equation}
for $\xi\in\mu(f)$.
\begin{proposition}
  \label{prop:bes_aux2}
  There exists $\nu_0>0$, with $\nu_0 = \nu_0(f,O)$, such that for
  every $\nu\geq\nu_0$ there is $\psi_\star>0$ such that
  \[
  \nu>\psi_\star + c(1+\tfrac1\nu)\phi_\star,
  \]
  and if
  \[
  \sum_\xi \|\nabla\tildeU_\xi\|_{L^2}\leq\psi_\star
  \]
  then there is a unique solution
  $U\in\bes^2(\R;H^1_0(O))\cap\bes^{1/2,2}(\R;L^2(O))$ to
  problem~\eqref{eq:bes_nl_aux2}.
  
  Moreover,
  \begin{equation}
    \label{eq:bes_nl_unif3}
    \begin{aligned}
      &\|U\|_{B^{\frac12,2}(\R;L^2)} + \|U\|_{B^2(\R;H^1_0)}
      \leq c(\nu,\|f\|_{\bes^{1,2}(\R)})\psi_\star +
      c(\nu,\phi_\star,\|f\|_{\bes^{1,2}(\R)}),
      \\
      &\sum_\xi \|\nabla U_\xi\|_{L^2}\leq \psi_\star
    \end{aligned}
  \end{equation}
\end{proposition}
\begin{proof}
  Let $(U^{(N)})_{N\geq1}$ be the sequence of solutions
  to~\eqref{eq:bes_nl_aux} provided by
  Proposition~\ref{prop:bes_aux}. By~\eqref{eq:bes_nl_unif} it follows
  that $(U^{(N)})_{N\geq1}$ is bounded in $\bes^{1/2,2}(\R;L^2)$ and
  in $\bes^2(\R;H^1_0)$, hence there is a sub-sequence weakly
  convergent to a limit point
  $U\in\bes^{1/2,2}(\R;L^2)\cap\bes^2(\R;H^1_0)$. Since~\eqref{eq:bes_nl_aux2},
  weak convergence is enough to pass to the limit in the equation.
  Uniqueness follows as in Proposition~\ref{prop:bes_aux}, using the
  bound on the viscosity.
  
  We only have to identify $\nu_0$ and
  $\psi_\star$. From~\eqref{eq:bes_nl_unif2} it follows that
  \[
  \bigl(\nu - \psi_\star - c(1 +
  \tfrac1\nu)\phi_\star\bigr)\Bigl(\sum_\xi \|\nabla
  U_\xi\|_{L^2}\Bigr) \leq c(1+\nu)\phi_\star +
  c(1+\tfrac1\nu)^2(\phi_\star)^2 +
  c(1+\tfrac1\nu)\phi_\star\psi_\star,
  \]
  so everything boils down to show that for $\nu$ large enough there
  is $\psi_\star$ such that
  \[
  \frac{c(1+\nu)\phi_\star + c(1+\tfrac1\nu)^2(\phi_\star)^2 +
    c(1+\tfrac1\nu)\phi_\star\psi_\star}%
  {\nu - \psi_\star - c(1 + \tfrac1\nu)\phi_\star} \leq \psi_\star,
  \]
  that is
  \[
  \psi_\star^2
  - \bigl(\nu - 2c(1+\tfrac1\nu)\phi_\star\bigr)\psi_\star
  + c(1 + \nu)\phi_\star + c^2(1+\tfrac1\nu)^2\phi_\star^2
  \leq 0.
  \]
  It is elementary to verify that the above polynomial has two positive
  solutions for $\nu$ large enough.
\end{proof}
\begin{remark}
  Clearly, without the assumption on the size of $\nu$ in the previous proposition,
  one can still show existence of at least one solution to~\eqref{eq:bes_nl_aux2}.
  The size condition on $\nu$ is necessary only for proving uniqueness.
\end{remark}
%
\subsubsection{Third step: the fixed point argument}
Under the assumptions of Proposition~\ref{prop:bes_aux2} we have a
well defined map $\tildeU\mapsto U$, where $U$ is the solution to
problem~\eqref{eq:bes_nl_aux2}. Denote the map by $\map$, then it is
clear that any fixed point of $\map$ is a solution
to~\eqref{eq:bes_nl3} and hence to~\eqref{eq:bes_nl2}.
\begin{proof}[Proof of Theorem~\ref{thm:bes_nl}]
  Fix $\nu\geq\nu_0$, where $\nu_0$ is given in
  Proposition~\ref{prop:bes_aux2}.  We prove that the map $\map$ is a
  contraction on the set $\mathcal{X}$ of all
  $U\in\bes^2(\R;H^1_0(O))\cap\bes^{1/2,2}(\R;L^2(O))$ that verify the
  bounds~\eqref{eq:bes_nl_unif3}.
  
  The fact that $\map$ maps $\mathcal{X}$ into $\mathcal{X}$ clearly
  follows from Proposition~\ref{prop:bes_aux2}, so we only need to
  prove that $\map$ is a contraction. This is obtained as in the proof
  of uniqueness of Proposition~\eqref{prop:bes_aux}. Indeed, if
  $E:=\tildeU_1 - \tildeU_2$ and $D:=\map(\tildeU_1) -
  \map(\tildeU_2)$, then
  \[
  \im\xi E_\xi - \nu\Delta E_\xi + \sum_{\eta+\theta=\xi}
  (E_\eta\cdot\nabla)\,U_\theta^1 + (\tildeU_\eta^2\cdot\nabla)D_\theta
  + (E_\eta\cdot\nabla)\,w_\theta + \nabla Q_\xi=0,
  \]
  with a suitable $Q$. By multiplying by $\overline{D_\xi}$,
  integrating by parts, and summing over $\xi$ we get
  \[
  (\nu - \psi_\star)\|\nabla E\|_{\bes^2(L^2)} \leq \bigl(\psi_\star +
  c(1+\tfrac1\nu)\phi_\star\bigr)\|\nabla E\|_{\bes^2(L^2)}
  \]
  that is $\|\nabla E\|_{\bes^2(L^2)}\leq K_0 \|\nabla
  E\|_{\bes^2(L^2)}$, with 
  \begin{equation*}
    K_0 := \frac{\psi_\star +
      c(1+\tfrac1\nu)\phi_\star}{\nu - \psi_\star}.
  \end{equation*}
  Likewise we also have $\|D\|_{\bes^{\frac12,2}(L^2)}\leq \nu
  K_0^2\|\nabla E\|_{\bes^2(L^2)}$.  Finally, by multiplying by
  $\overline{D_\xi}$, dividing by $\|\nabla D_\xi\|_{L^2}$ and summing
  over $\xi$ we get
  \[
  \sum_\xi \|\nabla E_\xi\|_{L^2} \leq K_0 \sum_\xi \|\nabla
  E_\xi\|_{L^2}.
  \]
  In conclusion, if $\nu$ is large enough it follows that $K_0<1$ and $\nu K_0^2<1$
  and the map $\map$ is a contraction.
  
  In order to conclude the proof, we need to show that if
  $(U_\xi)_{\xi\in\mu(f)}$ is solution to~\eqref{eq:bes_nl3}, then the
  Besicovitch almost periodic vector field $U$ having Fourier
  coefficients $(U_\xi)_{\xi\in\mu(f)}$ is a weak solution
  to~\eqref{eq:bes_nl2}, namely that for every divergence-free
  $\varphi\in C^\infty_c(O;\R^3)$ and every $\xi\in\R$,
  \[
  \mean\Bigl[\Bigl(\scal{\partial_t U,\varphi} - \nu\scal{U,
    \Delta\varphi} - \scal{U, \bigl((U+w)\cdot\nabla\bigr)\,\varphi} -
  \scal{w, (U\cdot\nabla)\,\varphi} \Bigr)\e^{\im\xi t}\Bigr] = 0,
  \]
  which is an easy consequence due to the
  bounds~\eqref{eq:bes_nl_unif3}.
\end{proof}
\subsubsection{Final considerations}
\label{sec:final_considerations}
Apparently the assumption~\eqref{eq:bes_phi1} seems to be essential
for the proof in the Besicovitch setting to work. The technical
problem is essentially related to the term
\[
\sum_\xi \sum_{\eta+\theta=\xi} U_\xi\cdot (V_\eta\cdot\nabla)W_\theta
\]
which is of order three, although all bounds on $U$, $V$, $W$ are of
order two, if one works in the framework of $\bes^2$ spaces. Young's
convolution inequalities tell us that in general there is no
possibility to bound the above term under these assumptions. In terms
of the time variable, we are trying to bound the Navier-Stokes
nonlinearity \emph{over the whole} $\R$.

Another possibility would be to use the other a-priori estimate,
namely the bound in $\bes^{1/2,2}(\R;L^2)$, which plays no role in the
proof of Theorem~\ref{thm:bes_nl}, using for instance the results in
Section~\ref{ss:meaning}. \emph{This possibility is ruled out by the
  non-linear term}. In fact, in the standard case of Leray-Hopf weak
solutions one has a better knowledge of the time derivative and this
can be used for instance with the Aubin-Lions compactness lemma to
handle the non linear term.

Indeed the non-linearity reads in Fourier variables (in time) as a
convolution and, whatever is the spectrum $\sigma(f)$ of the flux, the
spectrum of the solution to the non-linear problem will have the
$\Z$-module $\mu(f)$ as its spectrum. In different words, the
non-linearity creates a full set of harmonic resonances in the time
frequency. The structure of $\Z$-modules in $\R$ shows that the only
possibility to use a bound on the derivatives (while obtaining a
useful information for all times $t\in\R$) is the periodic case,
previously studied in~\cite{Bei2005c}. Indeed it is easy to verify the
following result.
\begin{proposition}
  Let $G\subset\R$ be a $\Z$-module. Then, either $G = \kappa\,\Z$ for some $\kappa\in\R$
  or $G$ is dense in $\R$.
\end{proposition}
\bibliographystyle{amsplain}

\end{document}